\documentclass[11pt,hyp,]{nyjm}
\usepackage{hyperref}
\hypersetup{nesting=true,debug=true,naturalnames=true}
\usepackage{graphicx,amssymb,upref}

\let\<\langle
\let\>\rangle
\newcommand{\doi}[1]{doi:\,\href{http://dx.doi.org/#1}{#1}}

\let\uml\"

\usepackage{amsthm, overpic}
\pdfoutput=1

\author{Jeffrey S. Meyer}
\address{Department of Mathematics\\ 
California State University\\ 
San Bernardino, CA 92407}
\email{jeffrey.meyer@csusb.edu}

\author{Christian Millichap}
\address{Department of Mathematics\\ 
Furman  University\\ 
Greenville, SC 29613}
\email{Christian.Millichap@furman.edu}

\author{Rolland Trapp}
\address{Department of Mathematics\\ 
California State University\\ 
San Bernardino, CA 92407}
\email{rtrapp@csusb.edu}

\title[Arith. and Hidden Symmetries of Pretzel FAL Complements]{Arithmeticity and hidden symmetries\\ of fully augmented\\ pretzel link complements}

\keywords{hyperbolic link complement, arithmetic link, hidden symmetry.}

\subjclass[2010]{Primary: 57M25; Secondary: 57M27, 57M50.}

\DeclareMathAlphabet{\curly}{U}{rsfs}{m}{n}

\DeclareMathOperator{\N}{N}

\DeclareMathOperator{\PSL}{PSL}

\DeclareMathOperator{\tr}{tr}

\newtheorem{thm}{Theorem}[section]
\newtheorem{cor}[thm]{Corollary}

\newtheorem*{thmA}{Theorem \ref{mainthm1}}

\newtheorem{prop}[thm]{Proposition}
\newtheorem{lem}[thm]{Lemma}
\newtheorem{ques}[thm]{Question}

\theoremstyle{definition}
\newtheorem{rem}[thm]{Remark}

\newtheorem*{rmk}{Remark}

\theoremstyle{remark}

\renewcommand{\labelenumi}{(\roman{enumi})}

\def\1{\mathbf{1}}

\theoremstyle{plain}

\theoremstyle{remark}

\newtheorem{defn}[thm]{Definition}

\theoremstyle{plain}

\newtheorem{lemma}[thm]{Lemma}

\renewcommand{\labelenumi}{(\roman{enumi})}
\newcommand{\C}{\mathbb{C}}
\newcommand{\Q}{\mathbb{Q}}

\makeatletter
\def\moverlay{\mathpalette\mov@rlay}
\def\mov@rlay#1#2{\leavevmode\vtop{%
   \baselineskip\z@skip \lineskiplimit-\maxdimen
   \ialign{\hfil$\m@th#1##$\hfil\cr#2\crcr}}}
\newcommand{\charfusion}[3][\mathord]{
    #1{\ifx#1\mathop\vphantom{#2}\fi
        \mathpalette\mov@rlay{#2\cr#3}
      }
    \ifx#1\mathop\expandafter\displaylimits\fi}
\makeatother

\makeatletter
\let\@@pmod\pmod
\DeclareRobustCommand{\pmod}{\@ifstar\@pmods\@@pmod}
\def\@pmods#1{\mkern4mu({\operator@font mod}\mkern 6mu#1)}
\makeatother

\begin{document}

\begin{abstract}
This paper examines number theoretic and topological properties of fully augmented pretzel link complements. In particular, we determine exactly when these link complements are arithmetic and exactly which are commensurable with one another. We show these link complements realize infinitely many CM-fields as invariant trace fields, which we explicitly compute. Further, we construct two infinite families of non-arithmetic fully augmented link complements: one that has no hidden symmetries and the other where the number of hidden symmetries grows linearly with volume. This second family realizes the maximal growth rate for the number of hidden symmetries relative to volume for non-arithmetic hyperbolic 3-manifolds. Our work requires a careful analysis of the geometry of these link complements, including their cusp shapes and totally geodesic surfaces inside of these manifolds.


\end{abstract}

\maketitle 
\setcounter{tocdepth}{1}
\tableofcontents


\section{Introduction}

Every link $L \subset \mathbb{S}^{3}$ determines a link complement, that is, a non-compact $3$-manifold $M = \mathbb{S}^{3} \setminus L$.  If $M$ admits a metric of constant curvature $-1$, we say that both $M$ and $L$ are hyperbolic, and in fact, if such a hyperbolic structure exists, then it is unique (up to isometry) by Mostow--Prasad rigidity. Ever since the seminal work of Thurston in the early 1980s \cite{Thurston2}, we have known that links are frequently hyperbolic and a significant amount of research has been dedicated to understanding how their geometries relate to topological, combinatorial, and number theoretic properties of these links. In this paper, we further investigate these relationships for a  particularly tractable class of links known as \textit{fully augmented links} (FALs).  These links often admit hyperbolic structures that can be explicitly described in terms of combinatorial information coming from their respective link diagrams. See Figure \ref{fig:intropic} for two diagrams of FALs.

There has already been significant progress made in developing relationships between a FAL diagram and the geometry of the corresponding link complement. Geometric structures for these link complements that can be constructed via diagrams were first described by Agol and Thurston in the appendix of \cite{L2004} in 2004. Futer--Purcell  used this construction to determine diagrammatic conditions that guarantee a FAL is hyperbolic and computed explicit estimates on the geometry of the cusp shapes of FALs in \cite{FP2007}. One nice feature of hyperbolic FALs is that any hyperbolic link $L \subset \mathbb{S}^{3}$ can be constructed via Dehn surgery on a hyperbolic FAL. Futer--Purcell \cite{FP2007} exploited this relationship along with their cusp geometry estimates of FALs to show that highly twisted links coming from Dehn surgeries of FALs admit no exceptional surgeries.  Purcell also exploited these geometric structures on FALs in \cite{P2007} to determine explicit bounds on volumes and cusp volumes of hyperbolic FALs in terms of diagrammatic information. Since the mid 2000s, FALs and their generalizations have further been explored by Purcell \cite{P2008} \& \cite{P2010},  Flint \cite{F2017}, and Harnois--Olson--Trapp \cite{HOT2018}. 

Here, we address several open questions about number theoretic properties and commensurability classes of FALs in the context of \textit{fully augmented pretzel links} (pretzel FALs).  These links are an infinite subclass of FALs whose geometric decompositions admit additional properties that can be exploited. They are constructed by fully augmenting any pretzel link, that is, enclosing each twist region of a pretzel link with an unkotted circle, which we call a crossing circle, and removing all full-twists from each twist region. See Figure \ref{fig:intropic} for a diagram of a pretzel link $K$ and the corresponding pretzel FAL $\mathcal{F}$. We let $\mathcal{P}_{n}$ denote a pretzel FAL with $n$ crossing circles and with no twists going through any of the crossing circles. A diagram of $\mathcal{P}_{3}$ is depicted on the right side of Figure \ref{fig:intropic}. We set $M_n = \mathbb{S}^{3} \setminus \mathcal{P}_{n}$. We refer to a pair of pretzel FALs as \textit{half-twist partners} if they have the same number of crossing circles and differ  by some number of half-twists going through these crossing circles; we also refer to the corresponding link complements as half-twist partners. For instance, $\mathcal{F}$ and $\mathcal{P}_{3}$ in Figure \ref{fig:intropic} are half-twist partners. Half-twist partners frequently exhibit a number of common features, as we shall see in this paper.  Moving forward, we will assume that $n \geq 3$, which guarantees that  $M_n$ and all of its half-twist partners are hyperbolic.  We refer the reader to Section \ref{sec:geodecomp} for a more thorough description of pretzel FALs, their geometric decompositions, and their half-twist partners. 

\begin{figure}[ht]
	\centering
	\begin{overpic}[width = \textwidth]{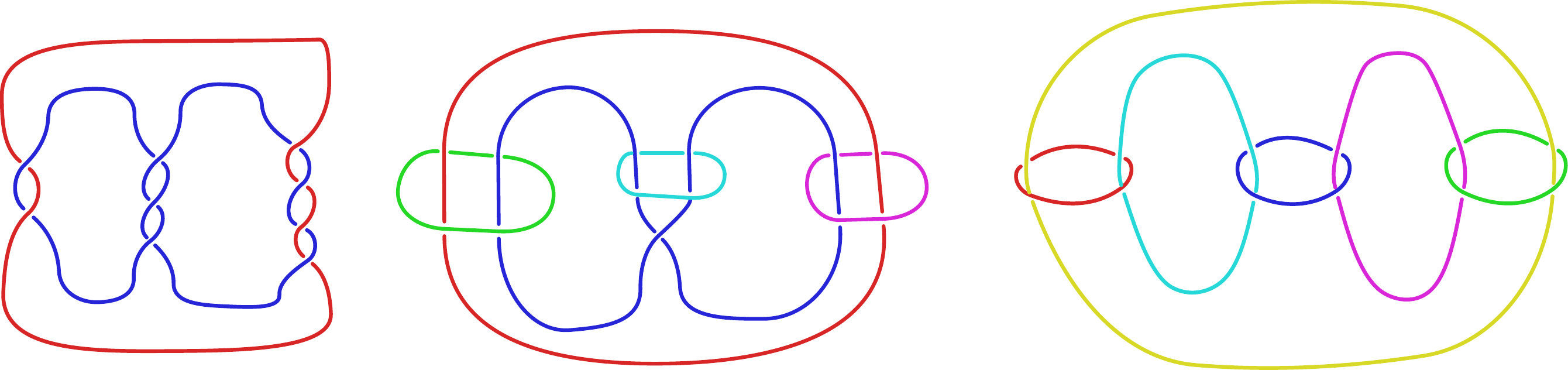}
		\put(2,22){$K$}
		\put(30,22){$\mathcal{F}$}
		\put(66,23){$\mathcal{P}_{3}$}
		\put(67,12){$c_{1}$}
		\put(81.5,12){$c_{2}$}
		\put(95,12.5){$c_{3}$}		
	\end{overpic}
	\caption{On the left is a diagram of a pretzel link $K$ with three twist regions. The middle diagram shows the pretzel FAL $\mathcal{F}$ obtained from fully augmenting $K$. The right diagram shows the pretzel FAL $\mathcal{P}_{3}$, which is a half-twist partner to $\mathcal{F}$.   Crossing circles of $\mathcal{P}_{3}$ are labeled by $c_i$, for $i=1,2,3$.}
	\label{fig:intropic}
\end{figure}
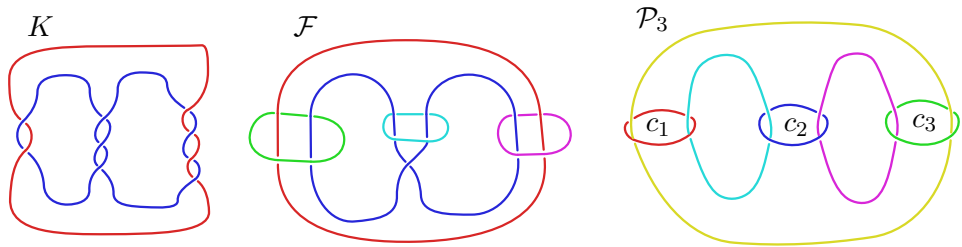

Our first major result is a complete classification of which pretzel FAL complements are arithmetic.
A link complement is \textit{arithmetic} if its fundamental group is commensurable to $\PSL_2(\mathcal{O}_k)$, where $k$ is some imaginary quadratic extension of $\Q$ and $\mathcal{O}_k$ is its ring of integers.
For more on arithmetic $3$-manifolds, we refer the reader to \cite{MR}.  
Arithmetic link complements have very restrictive topological and geometric properties, and in particular, they can not contain closed geodesics that are very short; see Theorem \ref{thm:NeRe} for an exact description. We show that most pretzel FAL complements (along with their half-twist partners and an even more general type of partner - see Definition \ref{defn:GeomPart}) are non-arithmetic by using their geometric decompositions to find short geodesics. The remaining cases are dealt with by examining properties of the invariant trace fields  for these link complements.  The following theorem shows that a pretzel FAL complement is non-arithmetic exactly when it has at least $5$ crossing circles in its respective diagram. See Section \ref{sec:Arith} for more details.

\begin{thm}
	\label{mainthm1}
	$M_{n}$ and all of its half-twist partners are arithmetic if and only if $n = 3,4$. 
\end{thm}

In Remark \ref{rem:M6} we highlight the curious case of $M_6$, whose arithmetic and geometric features resemble that of an arithmetic link complement, yet itself is not arithmetic.
As mentioned in the previous paragraph, part of the proof of Theorem \ref{mainthm1} requires comparing invariant trace fields of pretzel FAL complements. The \textit{invariant trace field} of a hyperbolic $3$-manifold $M = \mathbb{H}^{3} / \Gamma$ is the field generated by the traces of the products of squares of elements in $\Gamma$. Determining which fields are realized as invariant trace fields of hyperbolic link complements is a question of interest. Neumann has conjectured \cite{N2011} that every non-real number field arises as the invariant trace field of some hyperbolic $3$-manifold, yet to date, a relatively small collection of such fields have been verified to arise in this way.  

The invariant trace fields of arithmetic hyperbolic 3-manifolds are well understood \cite[Theorem 8.3.2]{MR}, and by cutting and gluing along thrice punctured spheres, any non-real multi-quadratic extension of $\Q$ can be realized  \cite[Theorem 5.6.4]{MR}.  Recently Champanerkar-Kofman-Purcell, assuming a conjecture of Milnor,  produced infinitely many incommensurable link complements each with invariant trace field $\Q(i,\sqrt{3})$ \cite{CKP2018}.  For certain subclasses of hyperbolic 3-manifolds, there are results restricting which invariant trace fields might arise (e.g. as for once-punctured torus bundles \cite[Thm. A]{Cal2006} or two-bridge knot complements \cite[Prop. 2.5]{RW2008}) or providing alternate characterizations (e.g. as for link complements \cite[Prop. 4.3]{NT2016}).

Here, we are  able to give an explicit description of the invariant trace fields of pretzel FAL complements by analyzing their cusp geometry. The work of Flint \cite{F2017} implies that the invariant trace field of a FAL complement is the same as its cusp field, that is, the field generated by all the cusp shapes of this link complement. In Section \ref{subsec:CA}, we compute the cusp shapes of these link complements and in Section \ref{section:tracefields} we analyze the properties of their invariant trace fields. The following theorem summarizes the major results from Section \ref{section:tracefields}.

\begin{thm}
\label{mainthmIVT}
The invariant trace field of any pretzel FAL with $n$ crossing circles is exactly  $\Q(\cos(\pi/n)i)$. In addition, there are only finitely many pretzel FAL complements with the same invariant trace field.
\end{thm}

These fields are particularly nice in that they are imaginary quadratic extensions of totally real number fields (i.e. CM-fields).  As of writing this, we are unaware of other hyperbolic 3-manifolds with these as their invariant trace fields for $n\ne 3,4,6$.  Additionally, as $n$ increases, the number of half-twist partners increases, thereby producing large collections of non-isometric manifolds realizing these invariants.

Our next major result examines commensurability classes of pretzel FAL complements. We say that two manifolds are \textit{commensurable} if they share a common finite-sheeted cover. The \textit{commensurability class} of a manifold $M$ is the set of all manifolds commensurable with $M$. It is usually difficult to determine if two hyperbolic $3$-manifolds are in the same commensurability class. Here, we determine exactly when two pretzel FAL complements  are commensurable with each other in terms of the number of crossing circles. To achieve this goal, we rely  on a fundamental result of Margulis \cite{Ma1991} which implies that if a hyperbolic $3$-manifold is non-arithmetic, then there exists a unique minimal orbifold in its commensurability class. For a non-arithmetic $M_n$, we show that it's minimal orbifold $\mathcal{O}_{n}$ is just the quotient of $M_n$ by a group of symmetries that are visually obvious in a particular diagram for these links; see Section \ref{sec:SymandHS} and Figure \ref{fig:SymDiag}. From here, we compute and compare the volumes of these minimal orbifolds, which help distinguish commensurability classes. We also determine commensurability relations for half-twist partners via a lemma from Chesebro--Deblois--Wilton \cite{CDW2012}. Finally, we deal with the few arithmetic pretzel FAL complements determined by Theorem \ref{mainthm1} on a case-by-case basis. 

\begin{thm}
	\label{mainthm2}
Suppose $M$ and $N$ are two hyperbolic pretzel FAL complements. Then $M$ and $N$ are commensurable if and only if they have the exact same number of crossing circles.
\end{thm}

An essential part of the proof of Theorem \ref{mainthm2} is showing that the minimal orbifold $\mathcal{O}_{n}$ is just the quotient of $M_n$ by a group of symmetries. In general, the minimal orbifold $\mathcal{O}$ in the commensurability class of a non-arithmetic hyperbolic $3$-manifold $M = \mathbb{H}^{3} / \Gamma$ is  $\mathcal{O} = \mathbb{H}^{3} / C(\Gamma)$, where $C(\Gamma) = \left\lbrace g \in \text{Isom}(\mathbb{H}^{3}) : | \Gamma : \Gamma \cap g \Gamma g^{-1} | < \infty \right\rbrace$ denotes the commensurator of $\Gamma$. Note that, $\Gamma \subset N(\Gamma) \subset C(\Gamma)$, where $N(\Gamma)$ denotes the normalizer of $\Gamma$. Elements of $N(\Gamma)/\Gamma$ correspond with symmetries of $M$, while elements of $C(\Gamma) / N(\Gamma)$ correspond with what we call \emph{hidden symmetries} of $M$. Thus, our proof of Theorem \ref{mainthm2} required us to show that $M_n$ has no hidden symmetries. This task is accomplished by showing that if $M_n$ did have hidden symmetries, then $\mathcal{O}_{n}$ would be a low volume, single-cusped, hyperbolic $3$-orbifold. Such orbifolds are either arithmetic or have very specific cusp shapes. In Section \ref{subsec:CA}, the necessary cusp shape analysis is provided to help eliminate the possibility of $M_n$ having hidden symmetries. 

Margulis's work \cite{Ma1991} mentioned above  can also be stated as a classification of arithmetic manifolds in terms of hidden symmetries: a hyperbolic $3$-manifold is arithmetic if and only if it has an infinite number of hidden symmetries.
Thus, it is natural to ask: how many hidden symmetries could a non-arithmetic hyperbolic $3$-manifold have? It is known, for example, that non-arithmetic two-bridge knot complements \cite[Thm. 3.1]{RW2008} and link complements \cite[Thm. 1.1]{MW2016} have no hidden symmetries. At the same time, there are a few  examples and methods for constructing links admitting hidden symmetries;  \cite{CD2017} gives one such construction.  However, can we find examples of non-arithmetic hyperbolic $3$-manifolds with the maximal number of hidden symmetries relative to volume? At most, the number of hidden symmetries of non-arithmetic hyperbolic $3$-manifolds could grow linearly with volume; see the remark after Corollary \ref{cor:countingHS}. In Section \ref{sec:HTPwithHS}, we construct examples exhibiting this optimal growth rate by considering particular half-twist partners of $M_n$, which we denote by $M_{n}'$; see Figure \ref{fig:HTPartner4}. To the best of our knowledge, there are no other explicit families of non-arithmetic hyperbolic $3$-manifolds realizing this growth rate described in the literature.  It is interesting to see that among half-twist partners, which share a number of geometric and topological features, we can have as many and as few hidden symmetries as possible. This is highlighted in the following theorem. 

\begin{thm}
\label{mainthm3}
For $n \geq 5$, each $M_n$ has no hidden symmetries, while the half-twist partner $M_{n}'$ has $2n$ hidden symmetries. Furthermore, the number of hidden symmetries of $M_{n}'$ grows linearly with their volumes.  
\end{thm}


Showing that each $M_n'$ has $2n$ hidden symmetries first requires us to determine the symmetries of $M_{n}'$. This is a far more challenging task than determining the symmetries of $M_n$ since quotienting by a visually obvious group of symmetries no longer gives a low volume orbifold. Instead, a careful analysis of how a certain cusp of $M_{n}'$ can intersect totally geodesic surfaces in $M_{n}'$ is used to limit the number of possible symmetries. From here, we can bootstrap off the fact that $M_{n}$ and $M_{n}'$  cover the same minimal orbifold $\mathcal{O}_{n}$ and the degree of this covering map is the same in both cases since these manifolds have the same volume. Our analysis of symmetries and hidden symmetries of $M_n$  determine exactly the degree of this cover. In turn, the degree of this cover and the number of symmetries of $M_{n}'$ determine the number of hidden symmetries of $M_{n}'$.  

This paper is organized into six additional sections beyond the introduction. In Section \ref{sec:geodecomp}, we describe the geometric decomposition of pretzel FAL complements and their half-twist partners and provide an analysis of their cusp shapes. In Section \ref{sec:HRO}, we show that a pretzel FAL complement is commensurable with any of its half-twist partners. In Section \ref{section:tracefields}, we determine properties of the invariant trace fields of pretzel FAL complements; this work ultimately relies on the cusp shapes calculated earlier. Theorem \ref{mainthm1} is proved in Section \ref{sec:Arith}. In Section \ref{sec:SymandHS}, we determine the symmetries of $M_n$, prove that $M_n$ has no hidden symmetries, and  prove Theorem \ref{mainthm2}. Finally, in Section \ref{sec:HTPwithHS}, we analyze the symmetries and hidden symmetries of $M_{n}'$.

The authors acknowledge support from U.S. National Science Foundation grants DMS
1107452, 1107263, 1107367 ''RNMS: Geometric Structures and Representation Varieties`` (the
GEAR Network). We would also like to thank Dave Futer for his helpful suggestions. 


\section{Geometric decomposition of fully augmented pretzel links}
\label{sec:geodecomp}

In this section, we first describe how to construct a pretzel FAL on the level of link diagrams in Section \ref{subsec:diagrams}. We then describe how to build the complements of these links, both topologically and geometrically, in Section \ref{subsec:FALcomps}. Afterwards, we describe how to build a set of hyperbolic $3$-manifolds that are both topologically and geometrically similar to a FAL complement in Section \ref{subsec:GP}. Finally, in Section \ref{subsec:CA}, we analyze the cusp shapes of pretzel FAL complements.

\subsection{FAL diagrams}
\label{subsec:diagrams}

To construct a hyperbolic FAL, start with a diagram $D(K)$ of a link $K \subset \mathbb{S}^{3}$ that is \textit{prime}, \textit{twist reduced} with at least two twist regions, and \textit{nonspittable}; see \cite[Section 1]{FP2007} for appropriate definitions. We create a diagram for a FAL $\mathcal{F}$ from $D(K)$ by first adding a crossing circle around each \textit{twist region} in $D(K)$ (any maximal string of bigons arranged end to end in the link diagram or a single crossing). After augmenting $K$ by adding in the crossing circles, remove all full twists within each twist region. This leaves either no twists or a single half-twist in each twist region. The resulting link is our FAL $\mathcal{F}$, and Purcell \cite[Theorem 2.5]{P2011} shows that our assumptions on $D(K)$ imply that $\mathcal{F}$ is hyperbolic. See the diagrams in Figure \ref{fig:intropic} for an example of a link diagram $D(K)$ and its corresponding FAL diagram $D(\mathcal{F})$.  

A FAL decomposes into two sets of components: crossing circles (those added in the augmenting process) and knot circles (components coming from the original link $K$).  We refer to a crossing circle as \emph{twisted} if the two strands going through the crossing circle have a single half-twist.  Otherwise, the two strands have no twists, and we refer to the corresponding crossing circle as \emph{untwisted}.



For most of this paper, we will focus on FALs resulting from fully augmenting pretzel links.  Let $P_n$ denoted a pretzel link with $n$ twist regions, each with an even number of half-twists.  We let $\mathcal{P}_n$ denote the FAL resulting from fully augmenting $P_n$.  Thus $\mathcal{P}_n$ is a link with $2n$ components, half of which are (untwisted) crossing circles and the other half are knot circles. See Figure \ref{fig:PnPretzel1} for a diagram of $\mathcal{P}_{n}$. 


\begin{figure}[ht]
	\centering
	\begin{overpic}[width = \textwidth]{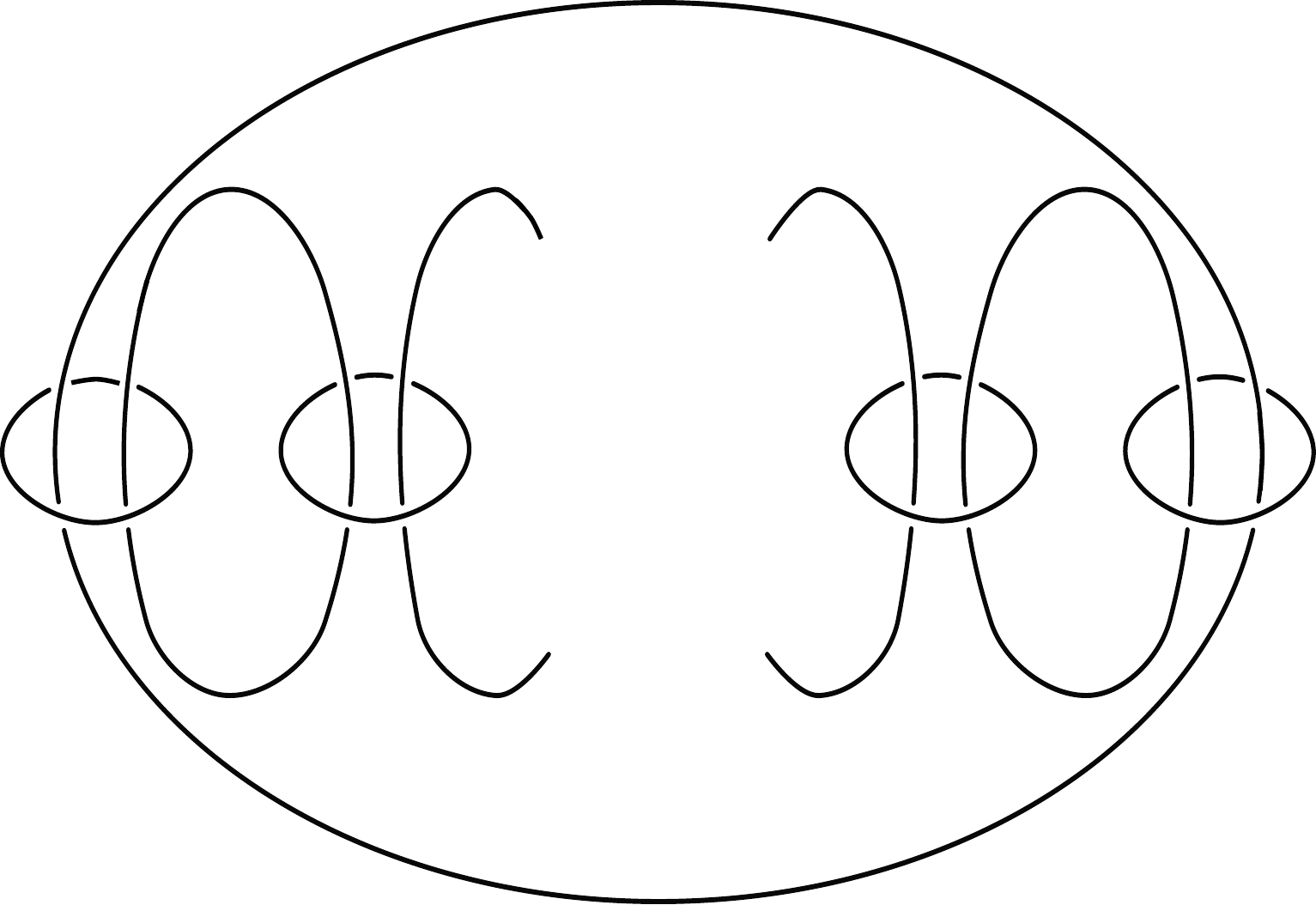}
		\put(1,28){$c_{1}$}
		\put(22,28){$c_{2}$}
		\put(62, 28){$c_{n-1}$}
		\put(86.5, 28){$c_{n}$}
		\put(47,19){\LARGE{\ldots}}
		\put(47,50){\LARGE{\ldots}}
	\end{overpic}
	\caption{$\mathcal{P}_{n}$, with crossing circles labeled $c_{i}$, $i = 1, \ldots, n$.}
		\label{fig:PnPretzel1}
\end{figure}
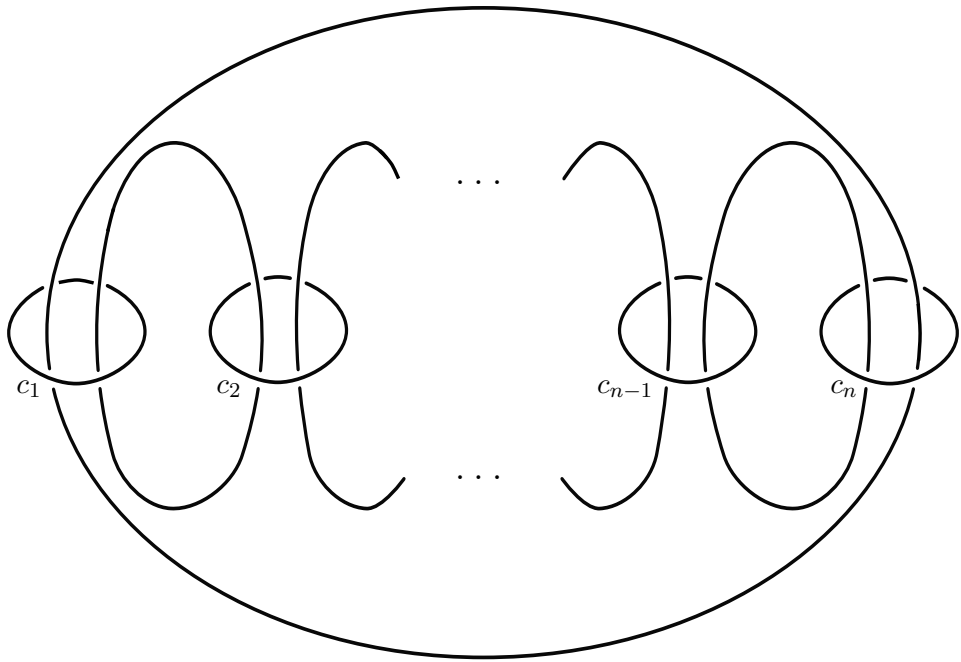


\subsection{FAL complements}
\label{subsec:FALcomps}

Given a hyperbolic FAL $\mathcal{F}$, consider the link complement $M_{\mathcal{F}} = \mathbb{S}^{3} \setminus \mathcal{F}$.  The following theorem collects a number of geometric properties of FALs coming from \cite{FP2007} and \cite{P2011}. In what follows, a \textit{crossing disk} is a twice-punctured disk whose boundary is a crossing circle of a FAL.

\begin{thm}
	\label{thm:FALstructure}
	Given a hyperbolic FAL $\mathcal{F}$,  its complement $M_{\mathcal{F}}$ has the following properties:
	
	\begin{enumerate}
		\item $M_{\mathcal{F}}$ decomposes into two identical ideal totally geodesic polyhedra, $P_{\pm}$, all of whose dihedral angles are right angles. 
		\item The faces of these polyhedra can be checkerboard colored shaded and white, where shaded faces are all triangles coming from crossing disks and white faces are portions of the projection plane bounded by the knot circles.  
		\item Intersecting a crossing disk with the projection plane creates two half disks. If we peel these two half disks apart, then four shaded faces are produced. Each of these four shaded faces is an ideal hyperbolic triangle, two in $P_{+}$ and two in $P_{-}$. 
	\end{enumerate}
\end{thm}

Here, we give a short summary of how to explicitly build the polyhedra $P_{\pm}$ and how to glue these polyhedra together to form a FAL complement. We will examine this decomposition in terms of our pretzel FAL complements,  $M_{n} = \mathbb{S}^{3} \setminus \mathcal{P}_{n}$, though this decomposition holds more generally for hyperbolic FALs. Examples where a FAL has some twisted crossing circles will be discussed in Section \ref{subsec:GP}. For more explicit details on geometric decompositions of FALs, we refer the reader to \cite{FP2007} and \cite{P2011}. By cutting $\mathbb{S}^{3} \setminus \mathcal{P}_{n}$ along the projection plane, we subdivide our link complement into two identical $3$-balls, one above the projection plane and one below. The crossing circles lie perpendicular to the projection plane, and so, each crossing disk is sliced in half by this process. To obtain our checkerboard coloring, peel each half crossing disk apart to obtain four shaded faces, two in each $3$-ball. By shrinking the link components to become ideal vertices, we obtain the ideal polyhedra $P_{\pm}$, each with the desired checkerboard coloring, which is depicted in Figure \ref{fig:polydecomp}. Note that, each crossing circle $c_{i}$ from our link diagram produces two shaded triangular faces, $c_{i}^{a}$ and $c_{i}^{b}$, in one of our polyhedra. In addition, there are $n+2$ white faces, one for each region of the projection plane, which are labeled $W_{i}$, for $i = 1, \ldots, n+2$. To reverse this process, we first glue each white face in $P_{+}$  to the white face in $P_{-}$ that corresponds with the same white face in the projection plane. We then glue up pairs of shaded faces corresponding to the same half crossing disk in the same polyhedra.

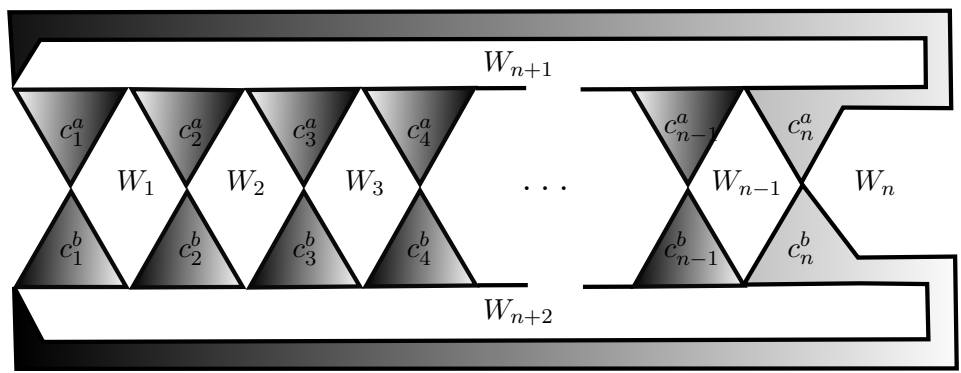
\begin{figure}[ht]
	\centering
	\begin{overpic}[width = \textwidth]{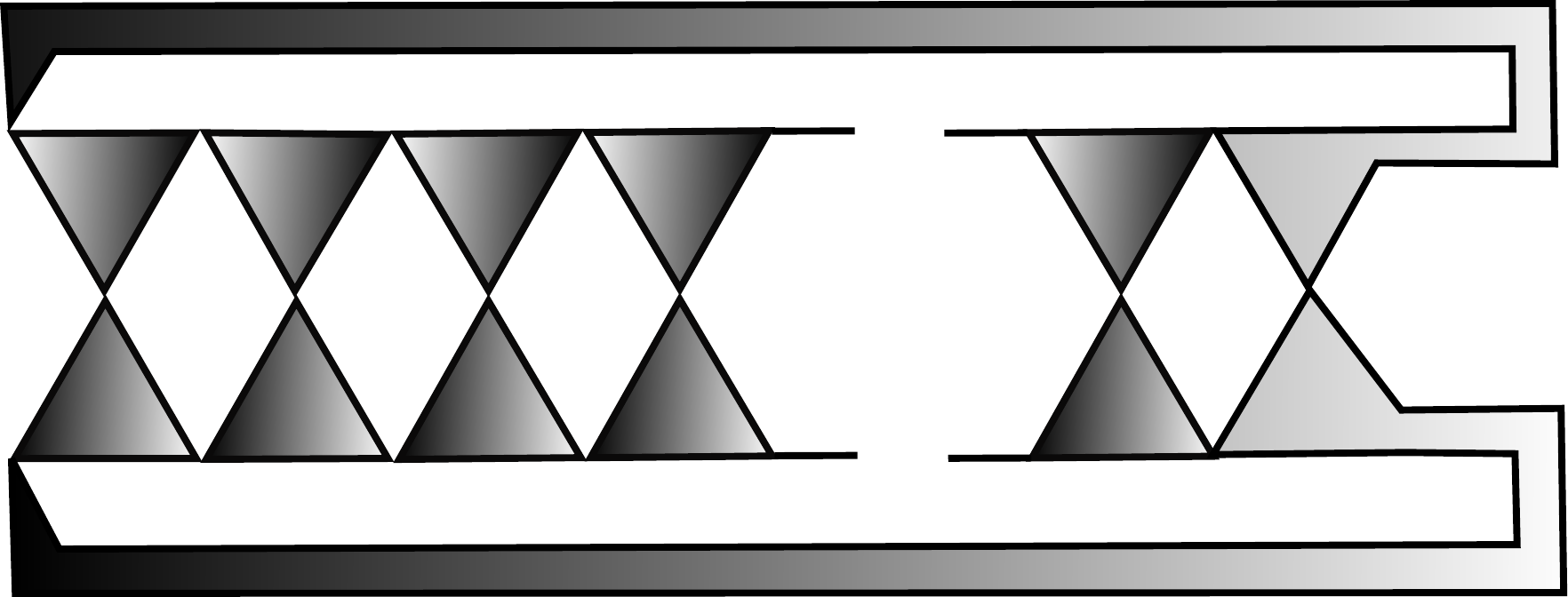}
		\put(5.5,24.5){$c_{1}^{a}$}
		\put(5.5,12){$c_{1}^{b}$}
		\put(18, 24.5){$c_{2}^{a}$}
		\put(18, 12){$c_{2}^{b}$}
		\put(30, 24.5){$c_{3}^{a}$}
		\put(30, 12){$c_{3}^{b}$}
		\put(42, 12){$c_{4}^{b}$}
		\put(42, 24.5){$c_{4}^{a}$}
		\put(69, 12){$c_{n-1}^{b}$}
		\put(69, 25){$c_{n-1}^{a}$}
		\put(82, 12){$c_{n}^{b}$}
		\put(82, 25){$c_{n}^{a}$}
		\put(11.5, 19){$W_{1}$}
		\put(23, 19){$W_{2}$}
		\put(35.5, 19){$W_{3}$}
		\put(74, 19){$W_{n-1}$}
		\put(89, 19){$W_{n}$}
		\put(50, 31.5){$W_{n+1}$}
		\put(50,5.5){$W_{n+2}$}
		\put(54,19){\LARGE{\ldots}}
	\end{overpic}
	\caption{The checkerboard tiling of $\mathbb{S}^{2}$ associated to $\mathcal{P}_{n}$. This tiling determines the faces of the polyhedra $P_{\pm}$.}
		\label{fig:polydecomp}
\end{figure}

The previous paragraph just gives a topological description of $P_{\pm}$. We now describe a geometric description of $P_{\pm}$ as regular ideal polyhedra in $\mathbb{H}^{3}$ that are reflections of each other.  The checkerboard coloring of the faces described in Theorem \ref{thm:FALstructure} actually provides instructions on how to explicitly build these two polyhedra in $\mathbb{H}^{3}$. First, just consider the set of white faces. This set induces a circle packing of $\mathbb{S}^{2}$, where we draw a circle for each white face, and two circles are tangent to each other if the corresponding white faces share a vertex. The shaded triangular faces also induce a circle packing of $\mathbb{S}^{2}$, dual to the white circle packing. The circle packing for the white faces is given in Figure \ref{fig:PretzelCirclePacking}. These circles are given the same labels as their corresponding faces in the checkerboard tiling of Figure \ref{fig:polydecomp}. To build one of our polyhedra, extend these circles (both white and shaded) on $\mathbb{S}^{2} \cong \partial \mathbb{H}^{3}$ to become totally geodesic planes in $\mathbb{H}^{3}$. Cut away these planes to obtain $P_+$. The polyhedron $P_-$ can be obtained by reflecting $P_+$ across any of its white faces.   Faces on $P_{\pm}$ that are reflections of each other will be called \emph{corresponding} faces.

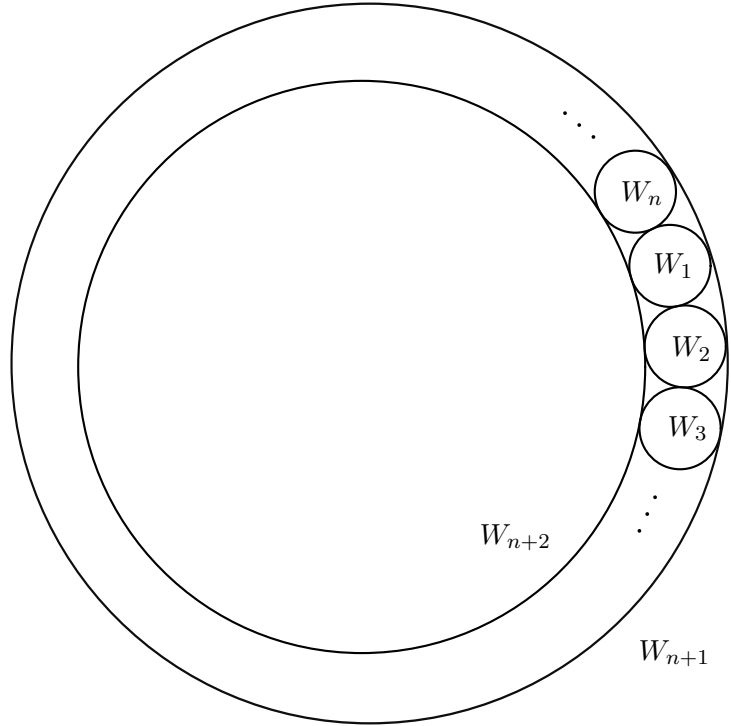
\begin{figure}[ht]
	\centering
	\begin{overpic}[scale=.80]{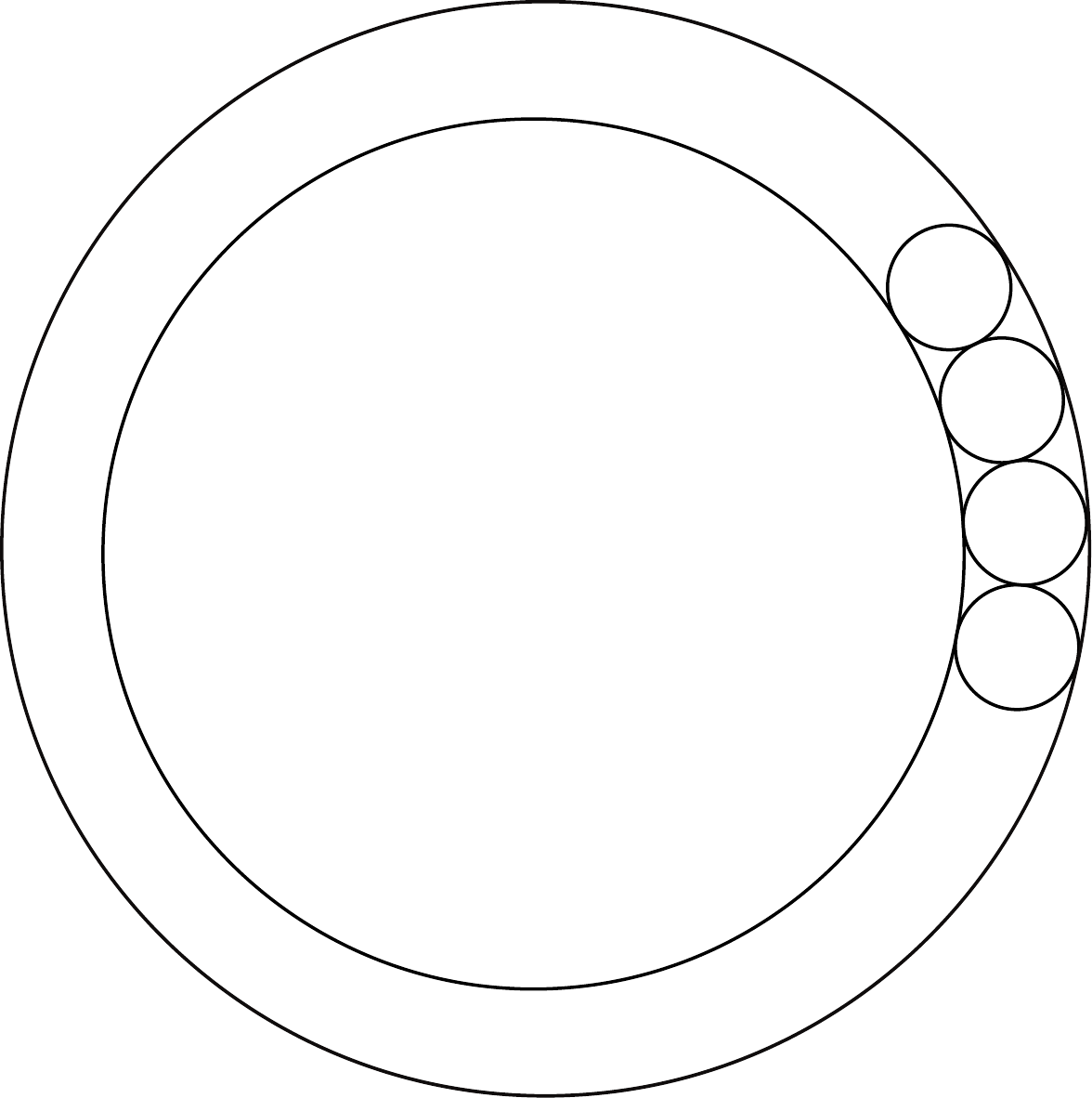}
		\put(89,62.5){$W_{1}$}
		\put(91.5,51){$W_{2}$}
		\put(91, 40){$W_{3}$}
		\put(84.5, 72.5){$W_{n}$}
		\put(65, 25){$W_{n+2}$}
		\put(87, 9){$W_{n+1}$}
		\put(76, 85){\rotatebox{-40}{\LARGE{\ldots}}}
		\put(86.5,26){\rotatebox{65}{\LARGE{\ldots}}}
	\end{overpic}
	\caption{The circle packing of $\mathbb{S}^{2}$ for the white faces of the polyhedral decomposition of $\mathcal{P}_{n}$.}
	\label{fig:PretzelCirclePacking}
\end{figure}

This is a convenient time to introduce a helpful combinatorial structure related to the circle packing.  It is used to provide examples of polyhedral partners for $M_{n}$ in Subsection \ref{subsec:GP} and to prove that $\mathcal{P}_n$ is commensurable with the reflection orbifold of $P_+$.  We proceed with the definition.

\begin{defn}\label{def:Crushtacean}
	The \textit{crushtacean} $C_{\mathcal{F}}$ of a FAL $\mathcal{F}$ is the trivalent graph whose vertices correspond to shaded faces of the circle packing, and an edge is inserted between two vertices if and only if the corresponding shaded faces share a vertex in the circle packing.
\end{defn}

The crushtacean was called the dual to the nerve of the circle packing by Purcell in \cite{P2011}.  It was first called the crushtacean in Chesebro--Deblois--Wilton \cite{CDW2012} and named so because it can be constructed by crushing the shaded faces of the checkerboard tiling of $\mathbb{S}^{2}$ associated to $\mathcal{P}_{n}$ down to vertices. This graph is trivalent since all shaded faces in this checkerboard tiling of $\mathbb{S}^{2}$ are triangles. In Figure \ref{fig:Crushtacean}, we show the crushtacean associated to our pretzel FAL $\mathcal{P}_{n}$, which comes from crushing the shaded faces seen in Figure \ref{fig:polydecomp} down to points.

\begin{figure}[ht]
	\centering
	\begin{overpic}[scale=.80]{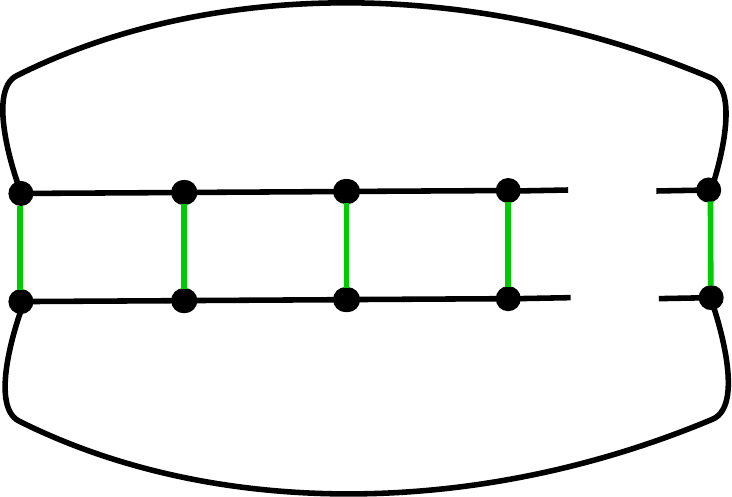}
		\put(77.5,33){\LARGE{\ldots}}
		\put(6, 32){$c_{1}$}
		\put(29, 32.5){$c_{2}$}
		\put(52, 32.5){$c_{3}$}
		\put(100, 33){$c_{n}$}
	\end{overpic}
	\caption{The crushtacean associated to the pretzel FAL $\mathcal{P}_{n}$. Edges coming from crossing circles have been labeled $c_{i}$ for $i =1, \ldots, n$ and colored green.}
	\label{fig:Crushtacean}
\end{figure}
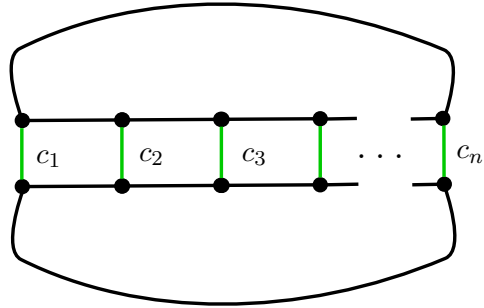


\subsection{Polyhedral Partners}
\label{subsec:GP}

In this section, we describe how to build sets of cusped hyperbolic $3$-manifolds that have a number of geometric and topological features in common with a FAL complement. To start, fix a FAL $\mathcal{F}$ that has $n$ crossing circles. Each crossing circle is either twisted or untwisted. By varying which crossing circles contain a half-twist, we can create $2^{n}$ FAL link diagrams that all differ by some number of half-twists. Some of these diagrams could correspond to the same link. However, it easy to show that many of these link (and their corresponding complements) will be different by considering the number of link components and their corresponding cusp shapes. Nevertheless, all of these link complements are built from the same two identical ideal totally geodesic polyhedra, $P_{\pm}$, just with gluing instructions modified. Gluing shaded faces across their common vertex on the same polyhedron results in an untwisted crossing circle while gluing a shaded face in $P_+$ to its mate in $P_-$ results in a twisted crossing cricle (see \cite[Figure 3]{P2011}). Since this is the only modification made between two FALs that differ by some number of half-twists, we would expect this set of FAL complements to be ``geometrically similar'' to one another. This all motivates the following definitions.

\begin{defn}	\label{defn:HTPart}
	Let $\mathcal{F}$ be a hyperbolic FAL with $n$ crossing circles, and fix an ordering on these crossing circles. At each crossing circle, assign a $0$ to designate an untwisted crossing circle and assign a $1$ to designate a twisted crossing circle. Define $\mathcal{F}' = \mathcal{F}(\epsilon_{1}, \epsilon_{2}, \ldots, \epsilon_{n})$ to be the hyperbolic FAL obtained from $\mathcal{F}$ by assigning $\epsilon_{i} \in  \left\lbrace 0, 1\right\rbrace$ to the $i^{th}$ crossing circle of $\mathcal{F}$. Any such $M_\mathcal{F'}$ is called a \textit{half-twist partner} of $M_\mathcal{F}$. We let $HTP(\mathcal{F})$ designate the set of all half-twist partners of $M_{\mathcal{F}}$.
\end{defn}

See the middle diagram of Figure \ref{fig:intropic} for a diagram of $\mathcal{P}_{3}(0,1,0)$. 

In fact, we can generalize the above definition and just consider cusped hyperbolic $3$-manifolds that are built from the polyhedra $P_{\pm}$ with gluing instructions modified along shaded faces.

\begin{defn}\label{defn:GeomPart}
	Let $M_{\mathcal{F}} = \mathbb{S}^{3} \setminus \mathcal{F}$ be a hyperbolic FAL with associated polyhedra $P_{\pm}$. We say that $M$ is a \textit{polyhedral partner} of $M_{\mathcal{F}}$ if  $M$ can be constructed from $P_{\pm}$ as follows:
\begin{enumerate}
	\item Corresponding white faces of $P_{\pm}$ are identified in the same manner as $M_{\mathcal{F}}$, and
	\item If $\varphi:G\to H$ identifies shaded faces $G$ and $H$, then their corresponding faces are identified by conjugating $\varphi$ with the reflection between $P_{\pm}$. 
\end{enumerate}
We let $PP(\mathcal{F})$ designate the set of all polyhedral partners of $M_{\mathcal{F}}$. 
\end{defn}	

\begin{rmk}\label{rmk:GeomPartners}
	Polyhedral partners are precisely the manifolds obtained using the admissible gluing patterns defined by Harnois-Olson-Trapp in \cite{HOT2018}.  Theorem 1 and Lemma 1 of that paper combine to show that, assuming corresponding white faces are glued without twisting (criteria (i) of Definition \ref{defn:GeomPart}), then criteria (ii) is necessary and sufficient to conclude that $M$ is hyperbolic.  Intuitively, polyhedral partners can be thought of as slicing $M_{\mathcal{F}}$ along the crossing disks, and then regluing in any manner that maintains a reflection surface.  

The set $PP(\mathcal{F})$ is a rich collection of manifolds, some of which are topologically well understood.  For convenience we include an example of a \emph{nested link} which is a polyhedral partner of $\mathcal{P}_5$; for more detail see \cite{HOT2018}. The combinatorial data that describes a nested link is an \emph{edge-symmetric spanning forest} of the crushtacean.  The left of Figure \ref{fig:NestedPartner} shows an edge-symmetric spanning forest of the crushtacean $\mathcal{C}_{\mathcal{P}_5}$.  Note that each tree in the spanning forest admits an involution swapping the endpoints of the ``middle" edge (hence the name edge-symmetric).  The shaded faces whose corresponding vertices are swapped under this involution are glued together.  The manifold resulting from this gluing is a \emph{generalized} FAL complement in $S^3$ as seen on the right of Figure \ref{fig:NestedPartner}.  Crossing circles of generalized FALs do not typically bound thrice punctured spheres.  The links constructed as above are termed \emph{nested} because the crossing circles nest in planes so that the regions between them form thrice punctured spheres.  The point of this discussion is that $PP(\mathcal{F})$ is a much broader class of links than FALs, and many of them have explicit topological descriptions.
\end{rmk}

\begin{figure}[ht]

	\begin{overpic}[width = 0.35\textwidth]{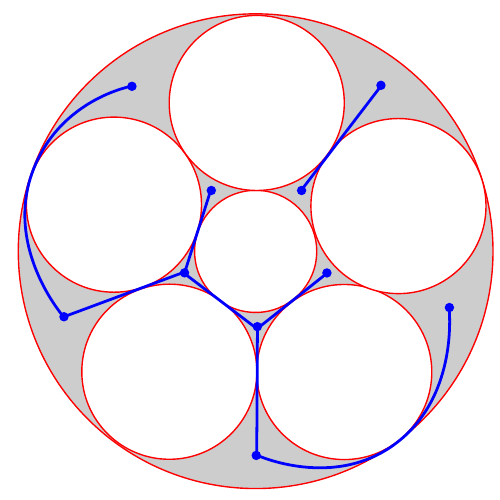} 
	\end{overpic}  \hspace{.25in} 
	\begin{overpic}[width = 0.55\textwidth]{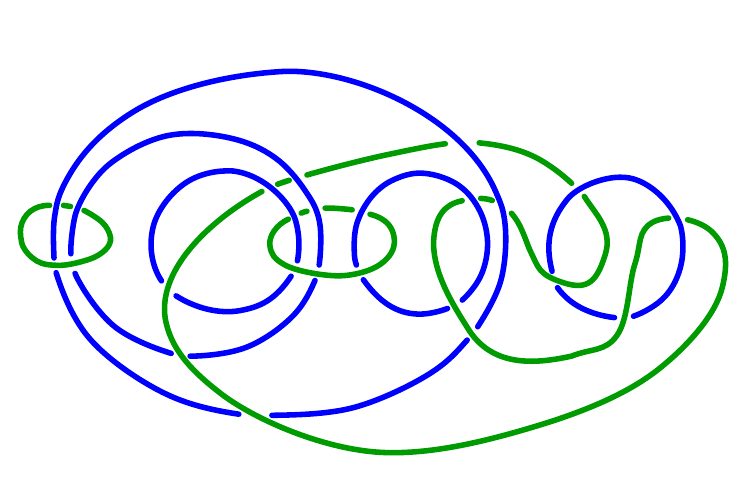}
	\end{overpic}
	\caption{An edge-symmetric forest and corresponding polyhedral partner of $\mathcal{P}_5$}
	\label{fig:NestedPartner}
\end{figure}


%


\subsection{Cusp Analysis}
\label{subsec:CA}

In this section we compute the cusp shapes for the link complements $M_n$.  The \textit{cusp shape} of a cusp $C$ of a cusped hyperbolic $3$-manifold $M$ is the Euclidean similarity class for the boundary torus  $\partial C$, which can be computed as the ratio of the meridian over the longitude on $\partial C$. 
The calculations of this section will later assist in analyzing arithmeticity, invariant trace fields, and hidden symmetries of $M_n$.

Since there is a symmetry of $M_{n}$ taking any component to any other, all cusp shapes are the same. We refer the reader to Figure \ref{fig:FAl3Pretz2} for a symmetric diagram of $\mathcal{P}_{n}$, where it is clear that each link component can be exchanged with its clockwise neighbor via a $90^{\circ}$ rotation of the entire chain followed by a rotation.  Symmetries of $(\mathbb{S}^{3}, \mathcal{P}_{n})$ are also symmetries of $M_n$ by Mostow--Prasad rigidity. Therefore, we isolate our attention to a single cusp corresponding to a crossing circle.   In general, such a cusp has torus boundary tiled by two identical rectangles, coming from intersecting $P_{\pm}$ with the horoballs centered at ideal vertices in $\partial \mathbb{H}^{3}$ corresponding to this crossing circle.  For a more thorough and general description of cusp shapes for FALs, see Lemma 2.3 and Lemma 2.6 of \cite{FP2007}.


Now recall that the circle packing for $P_+$ consists of a ring of $n$ circles nested between concentric circles (see Figure \ref{fig:PretzelCirclePacking}).  Assume that the ring of smaller circles are all unit circles, and consider the closer view given in Figure \ref{fig:PreInv}.  The shape of the cusp corresponding to $p$ will be determined.

\begin{center}
\begin{figure}[h]
\includegraphics[width=2in]{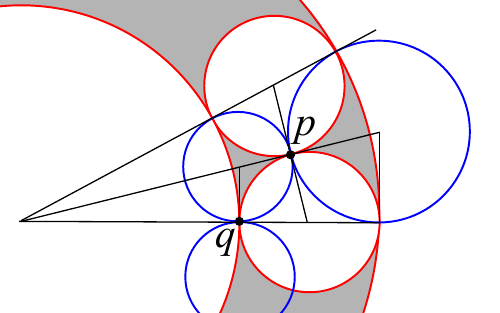}
\caption{Partial circle packing of $P_+$}
\label{fig:PreInv}
\end{figure}
\end{center}

Inverting in a unit circle $S$ centered at $p$ sends $p$ to infinity and the four faces of $P_+$ incident with $p$ will invert to intersect a horosphere in a rectangle $R$, whose shape can be explicitly calculated.  To do so, recall some elementary facts about inversion in planar circles.  

\begin{center}
\begin{figure}[h]
\includegraphics[width=2in]{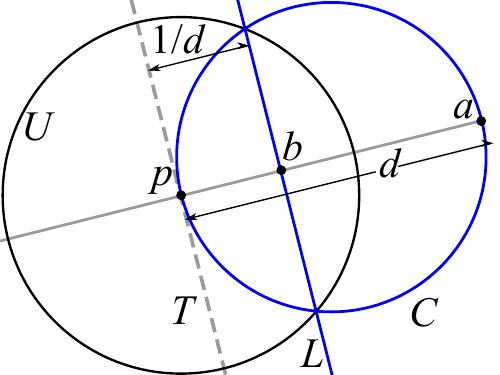}
\caption{Inverting circles to lines}
\label{fig:InvCircles}
\end{figure}
\end{center}

Let $U$ denote a unit circle centered at $p$. Let $C$ be any circle with diameter $d$ through $p$, and with tangent line $T$ there (see Figure \ref{fig:InvCircles}).  Inverting $C$ across $U$ yields a line $L$ parallel to $T$.  Moreover, the farthest point on $C$ from $p$ inverts to the closest point on $L$ (the points labeled $a$ and $b$ in Figure \ref{fig:InvCircles}).   Since $U$ is a unit circle, the distance from $L$ to $p$ is $1/d$.

\begin{figure}[h]
\[
\begin{array}{ccccc}
\includegraphics[width=1.25in]{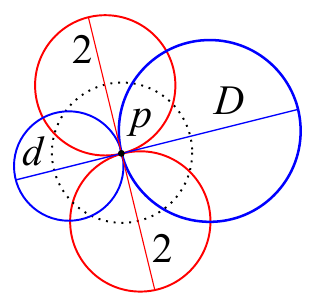} &\hspace{0.00in}& \includegraphics[width=1.7in]{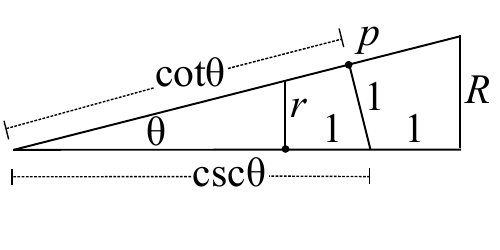} & \hspace{0.05in}&\includegraphics[width=1.25in]{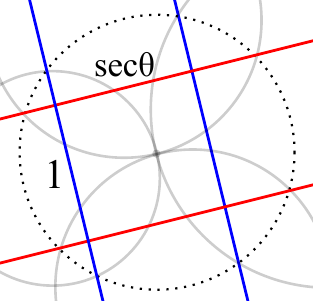}\\
(a) &\hspace{0.05in}& (b) &\hspace{0.0in}& (c)
\end{array}
\]
\caption{Inverting $p$ to infinity}
\label{fig:PreInvPCirc}
\end{figure}

We now return to the task of calculating the cusp shape for $\mathcal{P}_n$.

\begin{lemma}\label{lem:OneTile}
Let $P_+$ have $p$ at infinity, and $H$ be a horosphere centered at infinity. The shape of $R = P_+ \cap H$ is $i\sec(\pi/n)$.
\end{lemma}
\begin{proof}
Starting with our original $P_+$ as in Figure \ref{fig:PreInv}, we let $S$ be a unit sphere centered at $p$ and invert in $S$.  Focusing on those faces of $P_+$ incident with $p$, we get the diagram in Figure \ref{fig:PreInvPCirc}$(a)$ where the dotted circle is the boundary of $S$.  The two unit circles (diameter 2) bound white faces of $P_+$ and invert to parallel lines one unit apart by the remarks preceding the lemma.  The other circles bound shaded faces and invert to lines which are $\frac{1}{d} + \frac{1}{D}$ apart.  We now calculate the diameters $d$ and $D$.

To calculate the desired diameters, note that the centers of the unit circles are on a regular $n$-gon, and that the faint triangle in Figure \ref{fig:PreInv} can be labeled as in Figure \ref{fig:PreInvPCirc}$(b)$.  Indeed, all the edges labeled $1$ are radii of a unit circle in Figure \ref{fig:PreInv}.  Letting $\theta = \pi/n$, trigonometry labels the sides of the right triangle with right angle at $p$.  The edges labeled $r$ and $R$ are radii of the smaller and larger shaded circles (see Figure \ref{fig:PreInv}).  Using similar triangles one computes that $r=\tan\theta(\csc\theta - 1)$ and $R=\tan\theta(\csc\theta+1)$.  The diameters are twice that, and one computes
\[
\frac{1}{d} + \frac{1}{D} = \frac{1}{2\tan\theta(\csc\theta - 1)} + \frac{1}{2\tan\theta(\csc\theta + 1)} = \sec\theta.
\]
Thus inversion in $S$ sends $p$ to infinity and $P_+$ intersects a horosphere in a rectangle with white sides 1 unit apart and shaded sides $\sec(\pi/n)$ apart (see Figure \ref{fig:PreInvPCirc}$(c)$).

\end{proof}

Lemma \ref{lem:OneTile} determines the shape of $R$ near the cusp $p$.  To get a fundamental rectangle, we need to include $P_-$ near $p$.  We state the result as the following proposition. 


\begin{prop}\label{prop:CuspShape}
The cusp shape of each cusp of $M_{n}$ is $2\cos(\pi/n)i$. 
\end{prop}
\begin{proof}
As mentioned earlier, there is a symmetry of $M_n$ between any two components, so any two cusps will be isometric, for a particular choice of cusp expansion. This implies that all cusp shapes are isometric.  Thus we calculate the shape of the cusp $p$.  In Lemma \ref{lem:OneTile} we showed that $P_+$ near $p$ is a rectangle $R$ with shaded sides length 1 and white sides length $\sec\theta$.  Since $p$ is a crossing circle cusp, the shaded sides of $R$ are identified, while the white are not.  To get a fundamental region for the cusp we note that $P_-$ is the reflection of $P_+$ across any white side.  Thus a fundamental region is two tiles identical to that of Lemma \ref{lem:OneTile} glued along an white face, resulting in a $2$-by-$\sec\theta$ rectangle.  Rescaling to have unit meridian gives Figure \ref{fig:PnCuspShapes}$(a)$.  
\end{proof}

\begin{figure}[h]
\[
\begin{array}{cc}
\includegraphics[width=1.25in]{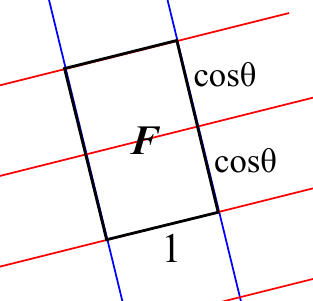} &\includegraphics[width=1.25in]{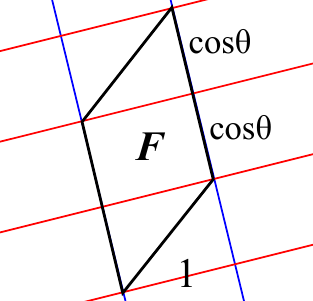}
\\
(a) \textrm{ Untwisted fundamental region $F$}& \hspace{0.5pc}(b) \textrm{ Twisted fundamental region $F$}
\end{array}
\]
\caption{Untwisted and twisted cusp shapes}
\label{fig:PnCuspShapes}
\end{figure}

\noindent\textbf{Remark:} We point out that since $p$ is a crossing circle, the meridian lies in the reflection surface and on the white sides of the fundamental region (labeled $\sec\theta$ in Figure \ref{fig:PreInvPCirc}$(c)$).  The longitude lies in the shaded faces.  If one normalizes the meridian to length 1, the corresponding longitude has length $2\cos(\pi/n)$.

Before moving on, we pause to consider the cusp shape of a half-twist partner of $M_n$.  The shape itself depends on whether the half-twist is positive or negative, but as the resulting manifolds are homeomorphic (they differ by a full-twist on the corresponding crossing disk) the distinction does not affect the invariants under consideration.

\begin{lemma}
	\label{lem:twistedtiles}
Let $C$ be a twisted crossing circle cusp in a half-twisted partner of $M_n$.  The cusp shape of $C$ is $\dfrac{2\cos(\pi/n)i}{1\pm \cos(\pi/n)i} $.
\end{lemma}

\begin{proof}
Changing from an untwisted to a twisted crossing circle alters the gluing pattern on the shaded faces.  For untwisted crossing circles, as in Proposition \ref{prop:CuspShape}, shaded faces are glued straight across, while for twisted crossing circles the gluing map shifts a single tile in the shaded direction (see \cite[Figure 7]{FP2007}).  Thus, scaling so that the distance between shaded faces is one, a fundamental parallelogram has longitude $2\cos(\pi/n)i$ and meridian $1\pm\cos(\pi/n)i$ (see Figure \ref{fig:PnCuspShapes}$(b)$ for the $1 + \cos(\pi/n)i$ case).  The result follows.
\end{proof}


\section{Hyperbolic reflection orbifolds}
\label{sec:HRO}




In this section, we establish a strong commensurability relation for our pretzel FAL complements - we show that any $M_{n}$, along with all of its half-twist partners, is commensurable with the \textit{hyperbolic reflection orbifold} associated with $P$ (a single copy of $P_{\pm}$). This is the orbifold $\mathcal{O}_{P} = \mathbb{H}^{3} / \Gamma_{P}$, where $\Gamma_{P}$ is generated by reflections in the faces of $P$. This commensurability relation will be used to help determine arithmeticity and hidden symmetries of half-twist partners in Section \ref{sec:Arith} and Section \ref{sec:HTPwithHS}, respectively.

While it perhaps seems probable that any FAL complement will be commensurable with its associated hyperbolic reflection orbifold, this is not always the case. For instance, Chesebro--DeBlois--Wilton in Section 7.2 of \cite{CDW2012} describe an infinite family of FALs that are not commensurable with any hyperbolic reflection orbifold. Fortunately, in the same paper, a  criterion is established to guarantee that  FAL complements with a certain combinatorial symmetry will be commensurable with their associated hyperbolic reflection orbifolds. This criterion is stated in terms of the crushtacean of Definition \ref{def:Crushtacean}. The following lemma is a rewording of \cite[Lemma 7.4]{CDW2012}; the commensurability conclusion stated in our version is noted in the paragraph following the proof of Lemma 7.4 in \cite{CDW2012}.

\begin{lemma}\cite[Lemma 7.4]{CDW2012}
	\label{lem:crush}
	Let $\mathcal{F}$ be a FAL with crushtacean $C_{\mathcal{F}}$. Suppose $C_{\mathcal{F}}$ has the property that for each crossing circle component $c_{i}$ of $\mathcal{F}$, corresponding to an edge $e_{i}$ of $C_{\mathcal{F}}$ with vertices $v_{i}$ and $v_{i}'$,
	\begin{enumerate}
		\item if $c_{i}$ is untwisted, then there is a reflective involution of $C_{\mathcal{F}}$ preserving $e_{i}$ and exchanging $v_{i}$ and $v_{i}'$.
		\item if $c_{i}$ is twisted, then there is a rotational involution of $C_{\mathcal{F}}$ preserving $e_{i}$ and exchanging $v_{i}$ and $v_{i}'$.
	\end{enumerate}
	Then $M_{\mathcal{F}}$ is commensurable with its associated reflection orbifold.  
\end{lemma}

\begin{prop}
	\label{prop:RefOrbComm}
	Suppose $M \in HTP(\mathcal{P}_{n})$. Then $M$ is commensurable with the reflection orbifold associated to $M_{n}$. In particular, if $M, M' \in HTP(\mathcal{P}_{n})$, then $M$ and $M'$ are commensurable. 
\end{prop}
\begin{proof}
The crushtacean for $P_{n}$ is given in Figure \ref{fig:Crushtacean}, where all of the edges coming from crossing circles are labeled $c_{1}, \ldots, c_{n}$. Note that, any half-twist partner of $P_{n}$ also has the same crushtacean with the same edge colorings. Consider the horizontal line intersecting all of the green edges of our crushtacean in their respective midpoints. Both reflecting across this line and rotating $180^{\circ}$ across this line provide involutions of our crushtacean that preserve any crossing circle edge $e_{i}$ while exchanging it's respective vertices. Thus, any $M \in HTP(\mathcal{P}_{n})$ satisfies the criteria of Lemma \ref{lem:crush}, and so, is commensurable with the reflection orbifold associated to $M_{n}$. The second statement follows from the fact that commensurability is an equivalence relation.   	
\end{proof}


\section{Cusp and trace fields}\label{section:tracefields}

The goal of this section is to prove Theorem \ref{mainthmIVT}, which is a key component for the proof of Theorem \ref{mainthm1}.  A complete hyperbolic 3-manifold $M$ of finite volume can be identified with the quotient $\mathbb{H}^3/\Gamma$ where $\Gamma\subset \PSL_2(\C)$ is a lattice.  The \textit{trace field} of $M$, denoted $\Q(\tr\Gamma)$, is the field generated by the traces of the elements of $\Gamma$.  Similarly, the \textit{invariant trace field} of $M$, denoted $kM$, is the field generated by the traces of the products of squares of the elements of $\Gamma$.  When $M$ is cusped, its \textit{cusp field},  denoted $cM$, is the field generated by the cusp shapes of $M$.   These fields are finite extensions of $\Q$ and satisfy the inclusions $cM\subset kM\subset \Q(\tr\Gamma)$ \cite{NR1990}.  In general, these fields are distinct, however in the case of FALs, by  \cite[Theorem 6.1.6]{F2017} and  \cite[Cor. 4.2.2]{MR}, these fields coincide, i.e. 
$cM=kM=\Q(\tr(\Gamma)).$  
This enables us to determine the invariant trace field of each $M_n$ and its half-twist partners.


\begin{prop}
	\label{prop:invtracfields}
	 If $M\in HTP(\mathcal{P}_n)$, $n \geq 3$, then $kM=\Q(\cos(\pi/n)i)$. 
\end{prop}

\begin{proof}
 By Proposition \ref{prop:CuspShape}, the cusp shapes of $M_n$ are all $2\cos(\pi/n)i$, and hence, $cM_n=\Q(\cos(\pi/n)i)$.
Flint \cite[Theorem 6.1.6]{F2017} shows $cM_n=kM_n$.  Since the invariant trace field is a commensurability invariant and half-twist partners are commensurable by Proposition \ref{prop:RefOrbComm}, it follows that $kM=\Q(\cos(\pi/n)i)$.
\end{proof}

Based on Proposition \ref{prop:invtracfields}, it's natural to ask if this result extends to polyhedral partners that are not half-twist partners. While half-twist partners always share the same invariant trace field (both by Flint's work \cite{F2017} and Theorem 5.6.1 from \cite{MR}), it does not seem immediately obvious that the techniques used in these papers would extend to polyhedral partners. 
This all motivates the following question. 

\begin{ques} How much can cusp and trace fields vary among polyhedral partners who are not half-twist partners?\end{ques}

Even though we have identified distinct primitive elements, $\cos(\pi/n)i$,  in each $kM_n$, the fields they generate may be isomorphic.  To differentiate them, we compute the degree of the field extension $[kM_n:\Q]$.

\begin{lem}\label{lem:degreephi}
For each $n\ge 3$, $[kM_n:\Q]=\phi(n)$, the Euler totient function.
\end{lem}

\begin{proof}
Since $(\cos(\pi/n)i)^2=-\cos^2(\pi/n)=-\frac{1}{2}(\cos(2\pi/n)+1)$, $kM_n$ is a quadratic extension of $\Q(\cos(2\pi/n))$.  Lehmer showed that $[\Q(\cos(2\pi/n)):\Q]=\phi(n)/2$  \cite{Leh} \cite{WatZeit1993}.  The result follows.
\end{proof}

In the proof of Lemma \ref{lem:degreephi}, we see that for each $n \geq 3$, $kM_n$ is an imaginary quadratic extension of its totally real subfield $\Q(\cos(2\pi/n))$, or in other words, is a CM-field.  In particular, $kM_n$ has no real places (cf. \cite[Thm. A]{Cal2006}).

\begin{lem}\label{lem:increasingdegree}
	For any $d\in \N$, there exists an $n_d\in \N$ such that for all $n\ge n_d$, $[kM_n:\Q]\ge d$.
\end{lem}

\begin{proof}
	This follows from the fact that $\phi(n)$ can be bounded from below by an increasing, unbounded function.  For example, it is known that for $n\ge 3$, $\phi(n)> \dfrac{n}{e^\gamma\log \log n +\frac{3}{\log \log n}}$, where $\gamma$ is Euler's constant \cite[Theorem 8.8.7]{BachShallit}.
\end{proof}

Putting these lemmas together, we get the following finiteness result.

\begin{prop}\label{prop:finitefields}For each $n\ge 3$, there are only finitely many $n_i$ such that $kM_{n_i}\cong kM_n$.\end{prop}

\begin{proof}
Suppose that for some $m\ge 3$, $kM_{m}\cong kM_n$.  Then $[kM_m:\Q]=[kM_n:\Q]$ and by Lemma \ref{lem:increasingdegree}, there are only finitely many $m$ for which this can hold.
\end{proof}

These results, taken together, now prove Theorem \ref{mainthmIVT}.  Meanwhile, Proposition \ref{prop:finitefields} motivates the following question: 

\begin{ques}\label{ques:ndeterminekMn} Does $m\ne n$ imply that $kM_m\not\cong kM_n$?\end{ques}

If $kM_m\cong kM_n$, then it is necessarily the case that $\phi(m)=\phi(n)$.  Since for a fixed $d\in \N$, $\phi(x)=d$ has only finitely many solutions, one strategy is to fix $d$ and then compare the fields $kM_n$ for each $n\in\phi^{-1}(d)$.  
In the low degree cases, direct computations are possible and the answer to Question \ref{ques:ndeterminekMn} is yes.   For example, $\phi(n)=2$ only when $n$ is $3$, $4$, or $6$.  A direct computation here shows that $kM_3=\Q(\sqrt{-1})$, $kM_4=\Q(\sqrt{-2})$, and $kM_6=\Q(\sqrt{-3})$.
More generally, implementing SAGE code, we are able to compute $kM_n$ for each $n\le 150$, and by comparing degrees and discriminants, verify than each are distinct.  


\begin{prop}
If $m,n\in \{3,4,\ldots, 150\}$ and $m\ne n$, then $kM_m\not\cong kM_n$.
\end{prop}

Unfortunately, the problem of understanding the solutions to $\phi(x)=d$ for large $d$ becomes quite complicated (see, for example, \cite{Ford1999}) which suggests that one should look for another strategy.
Should the answer to Question \ref{ques:ndeterminekMn} be yes, then this would supply an alternative proof to Corollary \ref{cor:comm}, below.



\section{Arithmeticity}
\label{sec:Arith}

In this section, we completely classify which pretzel FAL complements (along with their half-twist partners) are arithmetic. In order to do this, we employ two arguments: one argument uses short geodesics to rule out arithmeticity for most pretzel FAL complements and the other argument uses our invariant trace field calculations from Section \ref{section:tracefields} to take care of the remaining cases.

There are very strong restrictions placed on the possible lengths of short geodesics in a cusped arithmetic hyperbolic $3$-manifold, especially arithmetic link complements. Essentially, such manifolds rarely have short geodesics, and if they do, only a specific set of short lengths can be realized. This is highlighted in the following work from Neumann and Reid:

\begin{thm}\cite[Corollary 4.5, Theorem 4.6, Corollary 4.7]{NR1990}
	\label{thm:NeRe}
	Let $M = \mathbb{H}^{3} / \Gamma$ be a hyperbolic link complement. If $M$ is arithmetic and contains a geodesic of length less than $1.9248473002\ldots$, then $\Gamma$ is commensurable with $PSL_{2}O_{d}$ with $d \in \left\lbrace 1, 2, 3, 7, 11, 15, 19 \right\rbrace$ and the length of any such geodesic is one of the values from Table 1 (page 283 of \cite{NR1990}). If $M$ contains a geodesic of length less than $0.862554627$, then $M$ must be non-arithmetic. 
\end{thm}

Here, we will use Theorem \ref{thm:NeRe} to show that most $M_{n} = \mathbb{S}^{3} \setminus \mathcal{P}_{n}$ (and their respective polyhedral partners) are non-arithmetic. 

\begin{prop}
	\label{prop:NonarithMn}
	If $n \geq 7$, then $M_{n}$ and all of its polyhedral partners are non-arithmetic.   
\end{prop}

\begin{proof}
	We first show that $M_{n}$ has a sufficiently short geodesic when $n \geq 7$ and then apply Theorem \ref{thm:NeRe} to rule out arithmeticity. The fact that polyhedral partners are also non-arithmetic will immediately follow from how this geodesic is constructed. 

	Consider the circle packing used to build the polyhedra $P_{+}$ for $M_{n}$ shown in Figure \ref{fig:PretzelCirclePacking}. The circles $W_{n+1}$ and $W_{n+2}$ bound planes that contain faces of $P_+$. Let $\gamma^+$ be the vertical line segment through the origin that is the common perpendicular from $W_{n+1}$ to $W_{n+2}$.  The radii of the $W_i$ are needed to compute the length of $\gamma^+$, and Figures  \ref{fig:PreInv} and \ref{fig:PreInvPCirc}$(b)$ show that they can be chosen to be $\csc(\pi/n)\pm 1$.  The hyperbolic length of $\gamma^+$ is therefore $\ell(\gamma^+) = \ln\left(\frac{\csc(\pi/n)+1}{\csc(\pi/n) -1}\right)$.

To construct the short geodesic $\gamma$, let $P_-$ be the reflection of $P_+$ across the plane $H$ whose boundary is $W_{n+2}$ and let $\gamma^-$ be the reflection of $\gamma^+$.  The polyhedron $P_+\cup P_-$ is a fundamental region for $M_n$, with the innermost face (bounded by $W_{n+1}$) glued to the outermost (the reflection of the innermost across $H$) by a dilation.  Thus the curve $\gamma^+ \cup \gamma^-$ projects to a closed geodesic $\gamma$ in $M_n$.  Since the curves $\gamma^{\pm}$ are isometric, the length of $\gamma$ in $M_n$ is twice that of $\gamma^+$, or $\ell(\gamma) = 2\ln\left(\frac{\csc(\pi/n)+1}{\csc(\pi/n) -1}\right)$.

For $n \geq 15$, this creates a geodesic of length less than $0.862554627$, and so, by Theorem $\ref{thm:NeRe}$, any such $M_{n}$ is non-arithmetic. At the same time, for $7 \leq n  \leq 15$, we can compare $\ell(\gamma)$ to the geodesic lengths given in Table 1 on page 283 in the original statement of Theorem \ref{thm:NeRe} in \cite{NR1990}. Since $\ell(\gamma)$ does not match up with any of these values, we now know that $M_{n}$ is non-arithmetic for $n \geq 7$.
	
Finally, we note that the geodesic segments $\gamma^{+}$ and $\gamma^{-}$ run between white faces in their respective polyhedra, and never intersect the shaded faces. Thus, their union will always project to a geodesic of length $\ell(\gamma)$ in any polyhedral partner of $M_{n}$. Therefore, the same short geodesic analysis applied in the previous paragraph also applies to any polyhedral partner of $M_{n}$, as needed. \end{proof}

Proposition \ref{prop:NonarithMn} applies to $M_n$ and all of its polyhedral partners, making it applicable to a broad class of hyperbolic manifolds (see Remark \ref{rmk:GeomPartners}).  We now focus on the manifolds $M_n$ (and their half-twist partners), and classify which are arithmetic.  To do so, both Proposition \ref{prop:NonarithMn} and our invariant trace field calculations limit the possible values of $n$ to $3,4, 6$, and we use Vinberg's criteria to show $M_6$ is not arithmetic.  These techniques are more limited in scope, and do not immediately apply to polyhedral partners of $M_n$.

\begin{thmA}
	$M_{n}$ and all of it's half-twist partners are arithmetic if and only if $n = 3,4$. 
\end{thmA}

\begin{proof}
For a cusped, finite volume, hyperbolic manifold to be arithmetic, it is a necessary condition that its invariant trace field be an imaginary quadratic extension of $\Q$ \cite[Theorem 8.2.3]{MR}.  
By Proposition \ref{prop:invtracfields}, the invariant trace field of $M_n$ is $kM_n=\mathbb{Q}(\cos(\pi/n)i)$, which by Lemma \ref{lem:degreephi}, has degree $[kM_n:\Q]=\phi(n)$.  A straight forward computation shows that $\phi(n)=2$ implies that $n$ is 3, 4, or 6.  
Thus only $M_3$, $M_4$ and $M_6$ can be arithmetic. Additionally, Proposition \ref{prop:NonarithMn} also implies that only $M_n$ with $n \leq 6$ could be arithmetic. 

The fact that $M_3$ and $M_4$ are arithmetic was observed by Thurston in \cite[Chapter 6]{Thurston}, a fact that also follows from \cite[Lemma 7.6 and Corollary 7.5]{CDW2012} where they show that the crushtaceans of $M_3$ and $M_4$ imply they are arithmetic.

We now analyze the case of $M_6$.  Let $P$ denote  the hyperbolic polyhedra associated to $M_6$.  By Proposition \ref{prop:RefOrbComm}, $M_6$ is commensurable to the hyperbolic orbifold $\mathcal{O}=\mathbb{H}^3/\Gamma_P$ generated by reflections through the faces of $P$.  Since arithmeticity is a commensurability invariant, it suffices to analyze $\mathcal{O}$.  Associated to $P$ is the Gram matrix $G(P)$, which encodes the angles between intersecting faces, and distances between disjoint faces of $P$.  Vinberg's arithmeticity criterion in this context states that for $\mathcal{O}$ to be arithmetic, it is necessary that each entry in the Gram matrix is an algebraic integer \cite[10.4.5]{MR}.  A direct calculation shows this fails for the polyhedron $P$, which we now describe.

The Gram matrix entry corresponding to disjoint faces is $-2\cosh(\ell(\gamma))$ where $\ell(\gamma)$ is the length of the common perpendicular between the faces.  Let $\gamma^-$ be the common perpendicular between faces $W_{n+1}$ and $W_{n+2}$ of Figure \ref{fig:PretzelCirclePacking}. When $n=6$ the calculations of the proof of Proposition \ref{prop:NonarithMn} reduce to $\ell(\gamma^-) = \ln\left(\frac{\csc(\pi/6)+1}{\csc(\pi/6) -1}\right) = \ln 3$, and the corresponding Gram matrix entry is $-2\cosh(\ln 3) = -10/3$. Hence Vinberg's criterion fails, and $M_6$ is not arithmetic.

The extension to half-twist partners follows from Proposition \ref{prop:RefOrbComm}.
\end{proof}

\begin{rem}\label{rem:M6} The manifold $M_6$ is a rather intriguing example.  We know $kM_6=\Q(\sqrt{-3})$.  In \cite[Chapter 6]{Thurston}, Thurston remarked that the volume of $M_6$ is 20 times the volume of the figure-eight knot complement, the arithmetic knot complement whose invariant trace field happens to be $\Q(\sqrt{-3})$.  Despite these coincidences, $M_6$ is not arithmetic.  This suggests that $M_6$ should be considered in the future when looking for non-arithmetic manifolds that share other attributes with arithmetic manifolds.
\end{rem}

\section{Symmetries and hidden symmetries}
\label{sec:SymandHS}

Here, we completely classify the symmetries and hidden symmetries of pretzel FAL complements.  As a corollary, we are able to completely determine when pretzel FAL complements (and their half-twist partners)  are commensurable with one another. First, we define these terms and introduce some important tools necessary to understand why symmetries and hidden symmetries play a pivotal role in analyzing commensurability classes of hyperbolic $3$-manifolds.

 Given a hyperbolic $3$-manifold $M = \mathbb{H}^{3} / \Gamma$, the group of symmetries of $M$, denoted $Sym(M)$, is the group of self-homeomorphisms of $M$, up to isotopy. This group is homeomorphic to $N(\Gamma) / \Gamma$, where $N(\Gamma)$ denotes the normalizer of $\Gamma$ in $\text{Isom}(\mathbb{H}^{3})$. Let $Sym^{+}(M) \cong N^{+}(\Gamma) / \Gamma$ denote the subgroup of orientation-preserving symmetries, where $N^{+}(\Gamma)$ denotes the restriction of $N(\Gamma)$ to orientation-preserving isometries. To define a hidden symmetry of $M$, we first need to introduce the \textit{commensurator} of $\Gamma$, which is $C(\Gamma) = \left\lbrace g \in \text{Isom}(\mathbb{H}^{3}) : | \Gamma : \Gamma \cap g \Gamma g^{-1} | < \infty \right\rbrace$. We let $C^{+}(\Gamma)$ denote the restriction of $C(\Gamma)$ to orientation-preserving isometries of $\mathbb{H}^{3}$. 

Studying commensurators of $\Gamma$ is another way to approach studying the commensurability class of $M$: $M$ is commensurable with another hyperbolic $3$-manifold $N$ if and only if the corresponding commensurators of $M$ and $N$ are conjugate in  $\text{Isom}(\mathbb{H}^{3})$; see Lemma 2.3 of \cite{Wa2011}. Note that, $\Gamma \subset N(\Gamma) \subset C(\Gamma)$. \textit{Hidden symmetries} of $M$ correspond with nontrivial elements of $C(\Gamma)/ N(\Gamma)$. Geometrically, $M$ admits a hidden symmetry if there exists a symmetry of a finite cover of $M$ that is not a lift of an isometry of $M$. See \cite{Wa2011} for more background on hidden symmetries and commensurators of hyperbolic $3$-manifolds. 

Hidden symmetries also play a defining role in the arithmeticity of hyperbolic $3$-manifolds, and more generally, hyperbolic $3$-orbifolds. The work of Margulis \cite{Ma1991} shows that $C(\Gamma)$ is discrete in $\text{Isom}(\mathbb{H}^{3})$ with $\Gamma$ finite index in $C(\Gamma)$ if and only if $\Gamma$ is non-arithmetic.  Thus, $M = \mathbb{H}^{3} / \Gamma$ is arithmetic if and only if  $M$ has infinitely many hidden symmetries. Furthermore, in the non-arithmetic case, the hyperbolic $3$-orbifold $\mathcal{O} = \mathbb{H}^{3} / C(\Gamma)$ is the unique minimal orbifold in the commensurability class of $M$. This minimal orbifold (and its orientable variant: $\mathcal{O}^{+} = \mathbb{H}^{3} / C^{+}(\Gamma)$) will play an essential role in examining commensurability classes here. In particular, if $M$ admits no hidden symmetries, then $\mathcal{O} = \mathbb{H}^{3} / N(\Gamma) = M / Sym(M)$. In this case, we only need to determine the symmetries of $M$ to get our hands on the minimal orbifold in the commensurability class of $M$.

Now, we focus on determining the symmetries and hidden symmetries of our pretzel FAL complements. For our purposes, we will work with a symmetric diagram of $\mathcal{P}_{n}$; see Figure \ref{fig:FAl3Pretz2}. These links and their symmetric diagrams were examined in Example 6.8.7 of Chapter 6 of Thurston's notes \cite{Thurston}. In this setting, Thurston used $D_{2n}$ to denote our $\mathcal{P}_{n}$.   

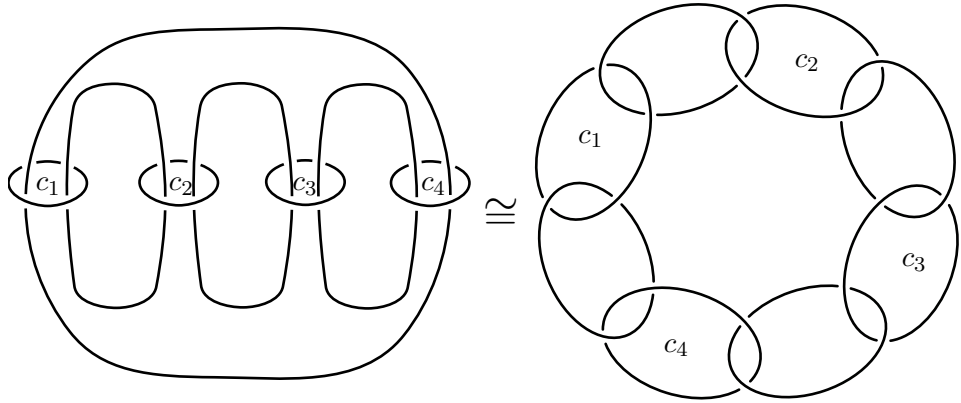
\begin{figure}[ht]
	\centering
	\begin{overpic}[width = \textwidth]{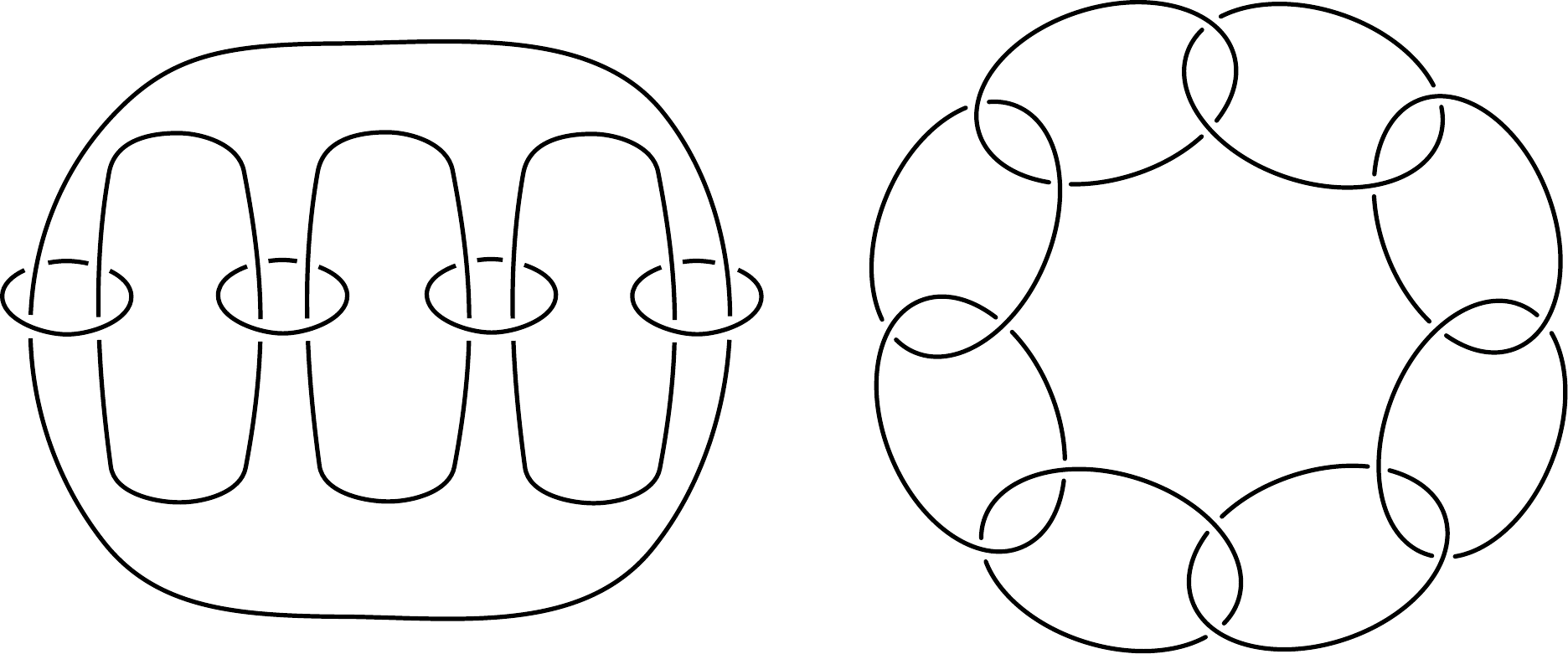}
		\put(50,18){\LARGE{$\cong$}}
		\put(3,22){$c_{1}$}
		\put(17,22){$c_{2}$}
		\put(30,22){$c_{3}$}
		\put(43.5,22){$c_{4}$}
		\put(60, 27){$c_{1}$}
		\put(82.75,35){$c_{2}$}
		\put(94,14){$c_{3}$}
		\put(69,5){$c_{4}$}
	\end{overpic}
	\caption{Two diagrams of the same pretzel FAL, $\mathcal{P}_{4}$. The left diagram comes from augmenting a pretzel link and the right diagram is the symmetric diagram as described in Thurston's notes. Crossing circles are labeled in each diagram.}
		\label{fig:FAl3Pretz2}
\end{figure}

In what follows, for a link $K \subset \mathbb{S}^{3}$, we let $Sym(\mathbb{S}^{3}, K)$ denote the group of homeomorphisms of the pair $(\mathbb{S}^{3}, K)$, up to isotopy. We use $Sym^{+}(\mathbb{S}^{3}, K)$ to denote the restriction to orientation-preserving symmetries. First, we identify a visually obvious subgroup of $Sym^{+}(\mathbb{S}^{3}, \mathcal{P}_{n})$, as viewed from the symmetric diagram of $\mathcal{P}_{n}$. Let $\alpha$ be the symmetry that takes every link component to its (clockwise) neighbor, swapping each knot component with a crossing circle component.  Let $\beta$ be the $180^{\circ}$ rotation about the circular axis depicted in Figure \ref{fig:SymDiag}. Let $\gamma$ be the $180^{\circ}$ rotation about the linear axis depicted in Figure \ref{fig:SymDiag}. These three elements generate a group of orientation-preserving symmetries of order $8n$, which we denote by $G^{+}_{n}$. 


\begin{figure}[ht]
	\centering
	\begin{overpic}[scale=1.0]{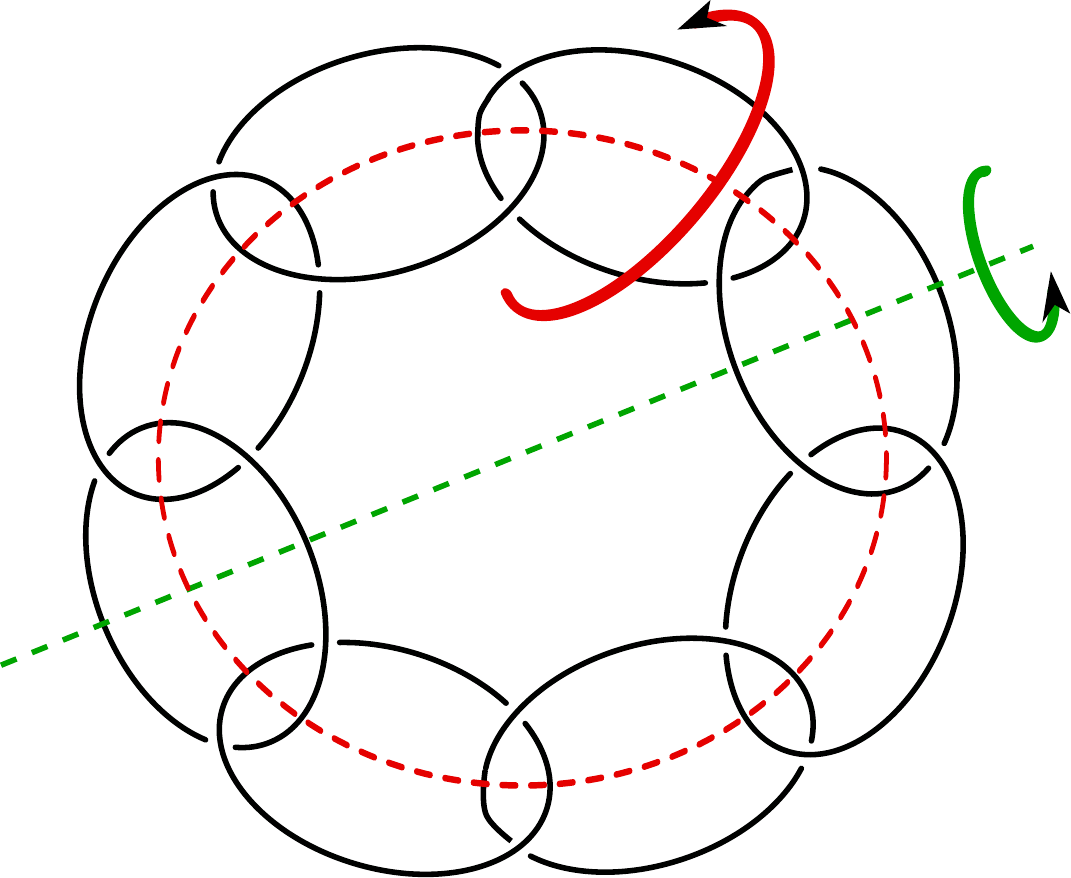}
		\put(73, 77){\LARGE{$\beta$}}
		\put(98,59){\LARGE{$\gamma$}}
	\end{overpic}
		\caption{A symmetric diagram for $\mathcal{P}_{4}$}
	\label{fig:SymDiag}
\end{figure}

Note that, $G^{+}_{n} \leq Sym^{+}(\mathbb{S}^{3}, \mathcal{P}_{n}) \leq Sym^{+}(M_{n})$, and both of these containments could be strict. By Mostow--Prasad Rigidity, the group of symmetries of a hyperbolic link in $\mathbb{S}^{3}$ is a subgroup of the symmetries of the corresponding link complement. For hyperbolic knots, these two groups are always equal by the work of Gordon--Luecke. However, for hyperbolic links with more than one component, this could be a strict containment; see \cite{HW1992} for such an example. At the same time, it is also possible that $\mathcal{P}_{n}$ has orientation-preserving symmetries beyond the ones we identified from its symmetric diagram. Our first goal is to show that we have identified all orientation-preserving symmetries,  and in fact, all the hidden symmetries of $M_{n}$.

\begin{thm}
	\label{thm:noHS}
	For any non-arithmetic $M_{n} = \mathbb{S}^{3} \setminus \mathcal{P}_{n}$, we have that $G^{+}_{n} = Sym^{+}(M_{n})$ and $M_{n}$ admits no (orientation-preserving) hidden symmetries. 
\end{thm}

To prove this theorem, we first need a lemma about the volumes of $M_{n}$.  Let $\mathcal{L} (\theta) = - \int_{0}^{\theta} \ln|2\sin(x)| dx$ denote the Lobachevsky function. In Example 6.8.7 in Chapter 6 of Thurston's notes \cite{Thurston}, the following volume formula is given: $vol(M_{n}) = 8n(\mathcal{L}(\frac{\pi}{4}+\frac{\pi}{2n}) + \mathcal{L}(\frac{\pi}{4} - \frac{\pi}{2n}))$. Let $f(n) = \frac{vol(M_{n})}{8n}$. This function will actually give us the volume of the minimal orbifold in the commensurability class of $M_{n}$. 

Let $v_{oct}$ denote the volume of a regular ideal hyperbolic octahedron. 

\begin{lemma}
\label{lemma:volprops}	
The function $f(n)$ is strictly increasing and $\lim_{n \rightarrow \infty} f(n) = \frac{v_{oct}}{4} \approx 0.915965$, for $n > 2$.  
\end{lemma}

\begin{proof}
A great exercise for your calculus students shows that $f'(n) = \frac{\pi}{2n^{2}} \ln | \frac{\sin(\frac{\pi}{4}+\frac{\pi}{2n})}{\sin(\frac{\pi}{4}-\frac{\pi}{2n})} |$. Notice that $f'(n) > 0$ if and only if $\frac{\sin(\frac{\pi}{4}+\frac{\pi}{2n})}{\sin(\frac{\pi}{4}-\frac{\pi}{2n})} > 1$. This second inequality holds if and only if $\sin(\frac{\pi}{4}+\frac{\pi}{2n}) > \sin(\frac{\pi}{4}-\frac{\pi}{2n})$. This inequality holds for $n>2$ since $\sin(\theta)$ is increasing on the interval $0 < \theta < \frac{\pi}{2}$. Thus, since $f'(n) > 0$ on our domain, we can conclude that $f(x)$ is strictly increasing on our domain. 

Also, we have that	 $\lim_{n \rightarrow \infty} f(n) = \lim_{n \rightarrow \infty} (\mathcal{L}(\frac{\pi}{4}+\frac{\pi}{2n}) + \mathcal{L}(\frac{\pi}{4} - \frac{\pi}{2n})) = 2\mathcal{L}(\frac{\pi}{4}) = \frac{v_{oct}}{4}$.
\end{proof}

Now, we prove Theorem \ref{thm:noHS}. In this proof, any symmetries or hidden symmetries will be assumed to be orientation-preserving. In what follows, a hyperbolic $3$-orbifold has a  \textit{rigid cusp} if it has a cusp whose cross section is  of the form $\mathbb{S}^{2}(2,4,4)$, $\mathbb{S}^{2}(3,3,3)$, or $\mathbb{S}^{2}(2,3,6)$. Likewise, a cusp of a hyperbolic $3$-orbifold is called a \textit{non-rigid cusp} if a cross section of this cusp is topologically either a torus or $\mathbb{S}^{2}(2,2,2,2)$.

\begin{proof}[Proof of Theorem \ref{thm:noHS}] 
 By Margulis's Theorem, we know that for any non-arithmetic $M_{n}$, there exists a unique minimal (orientation-preserving) orbifold in its commenusrability class, namely $\mathcal{O}^{+}_{n} = \mathbb{H}^{3} / C^{+}(\Gamma_{n})$. Let $Q^{+}_{n} = M_{n} / G^{+}_{n}$. If $M_{n}$ has any hidden symmetries or any symmetries beyond the ones contained in $G^{+}_{n}$, then $Q^{+}_{n} \neq \mathcal{O}^{+}_{n}$, and in particular, $Q^{+}_{n}$ is a non-trivial cover of $\mathcal{O}^{+}_{n}$. Since $|G^{+}_{n}| = 8n$, we have that $vol(Q^{+}_{n}) = \frac{Vol(M_{n})}{8n} = f(n)$. 	 Lemma \ref{lemma:volprops} implies that $vol(Q^{+}_{n}) < 0.915965$ for all non-arithmetic $M_{n}$. Since we are assuming $Q^{+}_{n}$ non-trivially covers $\mathcal{O}^{+}_{n}$, we have that $vol(\mathcal{O}^{+}_{n}) \leq \frac{vol(Q^{+}_{n})}{2} < 0.4579825$. We now consider two cases based on the cusp of $Q^{+}_{n}$. Note that, since $G^{+}_{n}$ contains a subgroup of symmetries exchanging all of the link components, the orbifold $Q^{+}_{n}$ only has one cusp. 
	
	\underline{Case 1}: Suppose the cusp of $Q^{+}_{n}$ is non-rigid. In this case, our volume bound guarantees that $\mathcal{O}^{+}_{n}$ is on the list of smallest volume (orientable) hyperbolic $3$-orbifolds with a non-rigid cusp highlighted in the work of Adams; see Corollary 4.2 and Lemma 7.1 of \cite{Adams1991}. In particular, all of these orbifolds are arithmetic. But $M_{n}$ is non-arithmetic and since arithmeticity is a commensurability invariant, this would imply that $\mathcal{O}^{+}_{n}$ is non-arithmetic, giving a contradiction.

	\underline{Case 2}: Suppose the cusp of $Q^{+}_{n}$ is rigid. In this case, the cusp field of $Q^{+}_{n}$ must be contained in $\mathbb{Q}(i)$ or $\mathbb{Q}(\sqrt{-3})$. The proof of Theorem \ref{mainthm1} shows that the invariant trace field of a non-arithmetic $M_{n}$ (which is the same as the cusp field of $M_{n}$) could be $\mathbb{Q}(i)$ or $\mathbb{Q}(\sqrt{-3})$ only if $n=6$. From here, we determine  $Sym^{+}(M_{6})$ via SnapPy and see that it has order $48$ (SnapPy actually determines the full symmetry group, which is order $96$). Since $|G_{6}| = 48$, we can conclude that $G_{6}$ contains all orientation-preserving symmetries of $M_{6}$. Now, suppose $M_{6}$ admits hidden symmetries. Then we have a non-normal cover of $\mathcal{O}^{+}_{6}$ by $Q^{+}_{6}$. Thus, $vol(\mathcal{O}^{+}_{6}) \leq \frac{vol(Q^{+}_{6})}{3} = f(6)/3 \approx  0.281928224$. However, this is impossible since this is smaller than the smallest volume for an orientable, one-cusped hyperbolic $3$-orbifold; see \cite{MM2012}. Thus, $M_{6}$ admits no hidden symmetries.
	
In conclusion, we must have that $Q^{+}_{n} = \mathcal{O}^{+}_{n}$, which implies that $M_{n}$ has no hidden symmetries and $G^{+}_{n} = Sym^{+}(M_{n})$. 
\end{proof}

There are a number of useful applications of Theorem \ref{thm:noHS}. First off, we can extend this same line of argument to determine $Sym(M_{n})$ and show that $M_{n}$ also admits no orientation-reversing hidden symmetries. 	Every $\mathcal{P}_{n}$ admits an orientation-reversing symmetry $\sigma$ given by reflection in the projection plane, and so, by Mostow--Prasad Rigidity, induces an orientation-reversing symmetry for $M_{n}$, which we shall also denote by $\sigma$. Let $G_{n}$ be the group generated by the elements of $G^{+}_{n}$ and $\sigma$.  This group has order $16n$. 

\begin{cor}
	\label{cor:fullsym}
	For any non-arithmetic $M_{n} = \mathbb{S}^{3} \setminus \mathcal{P}_{n}$, we have that $G_{n} = Sym(M_{n})$ and $M_{n}$ admits no hidden symmetries (both orientation-preserving and reversing). 
\end{cor}

\begin{proof}
 Consider the quotient $Q_{n} = M_{n} / G_{n}$.  If we suppose $Q_{n}$ is not the minimal orbifold in its respective commensurability class, then we can now apply the same argument used in the proof of Theorem \ref{thm:noHS}.
\end{proof}

Another nice application is the fact that we can now determine which pretzel FAL complements are commensurable with each other. The corollary given below describes this commensurability relation, and also, confirms Conjecture 6.2.6 from the work of Flint \cite{F2017}. 

\begin{cor}
	\label{cor:comm}
Suppose $M \in HTP(\mathcal{P}_{m})$ and $N \in HTP(\mathcal{P}_{n})$. Then $M$ and $N$ are commensurable if and only if $m =n$. 	

\end{cor}

\begin{proof} 
First off, Proposition \ref{prop:RefOrbComm} implies that if $m=n$, then $M$ and $N$ are commensurable. For the other direction, we break this proof down into a few cases, depending on whether or not $M_{n}$ and $M_{m}$ are arithmetic.

	\underline{Case 1}: Suppose $M_{n}$ and $M_{m}$ are non-arithmetic. By Margulis, there exists a unique minimal (orientable) hyperbolic $3$-orbifold  $\mathcal{O}^{+}_{n} = \mathbb{H}^{3} / C^{+}(\Gamma_{n})$ in the commensurability classes of $M_{n} = \mathbb{H}^{3} / \Gamma_{n}$. Likewise, we have that $\mathcal{O}^{+}_{m} =\mathbb{H}^{3} / C^{+}(\Gamma_{m}) $ is the minimal orbifold for $M_{m}$. By Theorem \ref{thm:noHS}, we know that $\mathcal{O}^{+}_{n}$ is just the quotient $M_{n} / G^{+}_{n}$, where $|G^{+}_{n}| = 8n$. Thus, $vol(\mathcal{O}^{+}_{n}) = \frac{vol(M_{n})}{8n} = f(n)$ and $vol(\mathcal{O}^{+}_{m}) = \frac{vol(M_{m})}{8m} = f(m)$. By Lemma \ref{lemma:volprops}, we know that $f(n)$ is strictly increasing, and so, we have that $vol(\mathcal{O}^{+}_{n}) \neq vol(\mathcal{O}^{+}_{m})$, whenever $n \neq m$. Thus, $\mathcal{O}^{+}_{n}$ and $\mathcal{O}^{+}_{m}$ are non-isometric, whenever $n \neq m$. Since this minimal orbifold is unique for each non-arithmetic commensurability class, this implies that $M_{n}$ and $M_{m}$ are not commensurable if $n \neq m$.

	\underline{Case 2}: Now, suppose $M_{n}$ and $M_{m}$ are both arithmetic. Then from Theorem \ref{mainthm1}, we know that $m, n \in \left\lbrace 3,4 \right\rbrace$.  Proposition \ref{prop:invtracfields} implies that $M_{3}$ and $M_{4}$ have different invariant trace fields, and so, they must belong to different commensurability classes. 
	
	\underline{Case 3}: Suppose $M_{n}$ is arithmetic, while $M_{m}$ is non-arithmetic. Since arithmeticity is a commensurability invariant, $M_{n}$ is not commensurable with $M_{m}$. 
	
	Thus, we have that $M_{n}$ is not commensurable to $M_{m}$, whenever $m \neq n$. The extension to half-twist partners follows from Proposition \ref{prop:RefOrbComm}.
\end{proof}

\section{Half-twist partners with many hidden symmetries}
\label{sec:HTPwithHS}

Here, we will analyze a special subclass of half-twist partners of pretzel FALs. Consider the pretzel FAL $\mathcal{P}_{n}$ and build it's half-twist partner $\mathcal{P}'_{n} = \mathcal{P}_{n}(0, 1, 1, 1, \ldots, 1)$. See the left side of Figure \ref{fig:HTPartner4} for the pretzel FAL diagram of $\mathcal{P}'_{5}$. In general, $\mathcal{P}'_{n}$ has $n$ crossing circles and $1$ knot circle. In addition, there is always exactly one untwisted crossing circle and $n-1$ twisted crossing circles in $\mathcal{P}'_{n}$. Set $M_{n}' = \mathbb{S}^{3} \setminus \mathcal{P}_{n}'$.

\begin{figure}[ht]
	\centering
	\begin{overpic}[width = \textwidth]{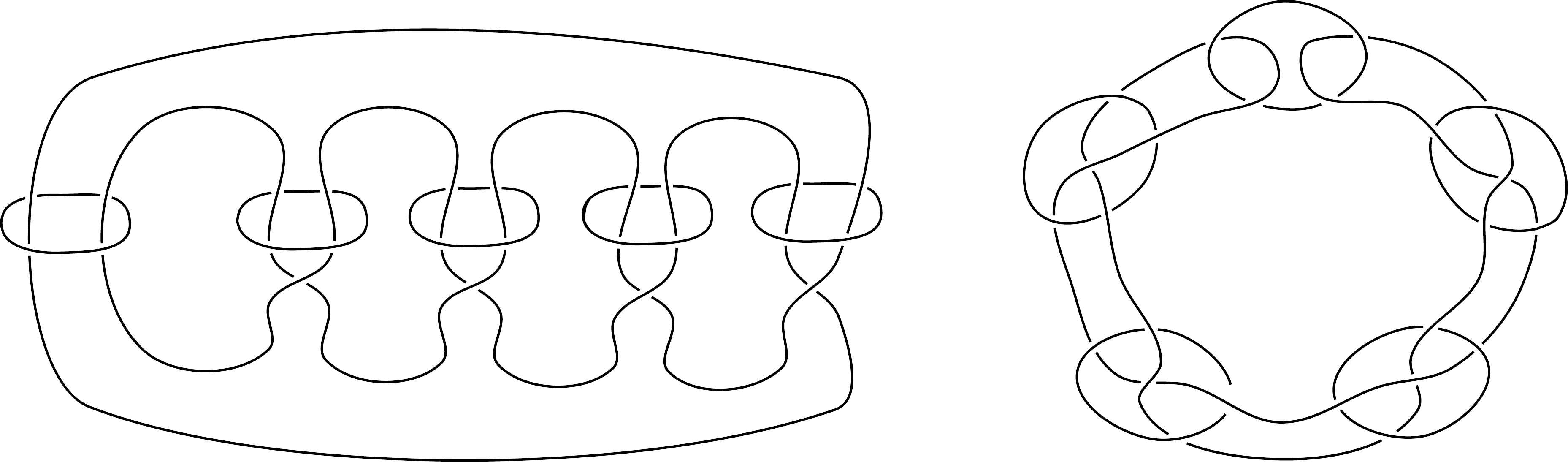}
		\put(59,15){\LARGE{$\cong$}}
	
	\end{overpic}
	\caption{Two link diagrams of $\mathcal{P}'_{5}$.}
	\label{fig:HTPartner4}

\end{figure}

Our main goal for this section is to show that each non-arithmetic $M_{n}'$ admits exactly $2n$ hidden symmetries, allowing us to construct non-arithmetic hyperbolic link complements with as many hidden symmetries as we would like, by taking $n$ sufficiently large.  To accomplish this goal, we first determine $Sym(M_{n}')$.  Similar to Section \ref{sec:SymandHS}, we start by identifying a visually obvious subgroup of symmetries of $(\mathbb{S}^{3}, \mathcal{P}_{n}')$ and then work to show this group must generate the full symmetry group $Sym(M_{n}')$. However, unlike in Section \ref{sec:SymandHS}, we can no longer use the fact that the quotient of $M_{n}'$ by this visually obvious subgroup of symmetries gives a small volume orbifold. Instead, we will analyze how cusps intersect certain totally geodesic surfaces in $M_{n}'$ in order to limit the number of symmetries of $M_{n}'$. From here, we can indirectly use our work from Section \ref{sec:SymandHS} to determine the number of hidden symmetries of $M_{n}'$.

By examining the diagram on the right side of Figure \ref{fig:HTPartner4}, we can see that $(\mathbb{S}^{3}, \mathcal{P}_{n}')$ admits three order two symmetries: $180^{\circ}$ rotation about the circular axis cutting through all of the half-twists (similar to the symmetry $\beta$ in Figure \ref{fig:SymDiag}), $180^{\circ}$ rotation about the line $L$ going through the middle of the untwisted crossing circle and the center of this ring of links, and the reflection in the vertical plane containing $L$.

\begin{thm}
	\label{thm:smallsymgroup}
	$Sym(M_{n}') \cong \mathbb{Z} / 2\mathbb{Z} \times \mathbb{Z} / 2\mathbb{Z} \times \mathbb{Z} / 2\mathbb{Z}$, generated by the symmetries mentioned above. 
\end{thm}

The proof of Theorem \ref{thm:smallsymgroup} requires some technical lemmas describing how symmetries of $M_{n}'$ could act on its cusps. We now proceed to state and prove these lemmas before returning to the proof of Theorem \ref{thm:smallsymgroup}. 

By abuse of notation, let $C$ denote both the untwisted crossing circle of $\mathcal{P}_{n}'$ and the corresponding cusp of $M_{n}'$.

\begin{lemma}
	\label{lem:symMaptoSelf}
	Let $\rho \in Sym(M_{n}')$. Then $\rho$ maps the  cusp $C$ to itself.
\end{lemma}

\begin{proof}
A symmetry $\rho$ can not map $C$ to a cusp coming from a crossing circle with half-twists since they have different cusp shapes; see Proposition \ref{prop:CuspShape} and Lemma \ref{lem:twistedtiles} for cusp shape descriptions. Now, we just need to show $C$ can not map to the knot circle cusp. A nice description of how to build the boundary torus of a knot circle cusp for a FAL complement is given in Lemma 2.3 and Lemma 2.6 of \cite{FP2007}. In particular, the larger the number of crossing circles a knot circle goes through, the longer the longitude of this cusp (relative to its meridian). In our case, our knot circle goes through all $n$ crossing circles twice, and so, a direct application of \cite[Lemma 2.6]{FP2007} shows that the length of the longitude of this knot circle cusp is at least $2n$, while its meridian is length exactly $2$ (for a particular horoball expansion).  At the same time, since the cusp shape of $C$ is $2\cos(\pi/n)i$, the ratio of the meridian to the longitude for $C$ is at most two-to-one.  Thus, $C$ could not map to the knot circle cusp. 
\end{proof}

Choose a cusp expansion for $M_{n}'$ and let $[\mu]$ and $[\lambda]$ denote the isotopy classes of the meridian and longitude, respectively, on the boundary torus $\partial C$. When we refer to the length of an isotopy class of a closed geodesic on $\partial C$, we mean the length of a geodesic representative. 

\begin{lemma}
	\label{lem:meridianslongitudes}
Let $\rho \in Sym(M_{n}')$. Then $\rho([\mu]) = \pm [\mu]$ and $\rho([\lambda]) = \pm [\lambda]$. 
\end{lemma}

\begin{proof}
 Given $\rho \in Sym(M_{n}')$, we know that $\rho$ maps $C$ to $C$ from Lemma \ref{lem:symMaptoSelf}. Since $\rho$ is an isometry, it must map geodesics on the torus $\partial C$ to  geodesics on $\partial C$ with the same length. Choose the cusp expansion for $C$ so that $[\mu]$ has length $1$ and $[\lambda]$ has length $|2\cos(\pi/n)|$, as done in the proof of Proposition \ref{prop:CuspShape}. Recall that all geodesics on the torus $\partial C$ are of the form $k_{1} \cdot \pm [\mu] + k_{2} \cdot \pm [\lambda]$ for some integers $(k_{1}, k_{2})  \neq (0,0)$.  Since $1 < |2\cos(\pi/n)|$,  we must have  $\rho([\mu]) = \pm [\mu]$.  Similarly, since $1 < |2\cos(\pi/n)|<2$, we must have that $\rho([\lambda]) = \pm [\lambda]$. 
\end{proof}

The above lemma only tells us that given any fixed cusp expansion, the isotopy classes of the meridian and longitude must map to themselves.  We would now like to place stronger restrictions on where particular geodesic representatives for $[\mu]$  and $[\lambda]$ could be mapped to under a symmetry of $M_{n}'$.

Now let $D$ be the untwisted crossing disk that $C$ bounds, and let $W$ be the reflection surface in $M'_n$ resulting from gluing the white faces of $P_{\pm}$.  We wish to show $D$ and $W$ are each fixed set-wise by $Sym(M'_n)$.

The main tool is an analysis of intersecting embedded totally geodesic surfaces, particularly when one of them is a thrice-punctured sphere.  Some preliminary observations are in order.  An embedded totally geodesic surface in a hyperbolic 3-manifold lifts to a union of disjoint planes in the universal cover.  Consequently, embedded totally geodesic surfaces intersect in a collection of pairwise disjoint simple geodesics.  There are six simple geodesics on a thrice-punctured sphere, three joining distinct cusps (\emph{intercusp} geodesics labeled $a, b, c$ in Figure \ref{fig:FixedD}$(a)$) and three from a cusp to itself (\emph{intracusp} geodesics  labeled $x, y, z$ in Figure \ref{fig:FixedD}$(a)$).  Figure \ref{fig:FixedD}$(b)$ displays these geodesics on a lift $\widetilde{D}$ to a fundamental region for $D$ in $P_+\cup P_-$.   Any embedded totally geodesic surface that intersects $D$ must intersect $D$ in a pairwise disjoint subset of these geodesics.

\begin{center}
\begin{figure}[h]
\[
\begin{array}{cc}
\includegraphics[width=1.75in]{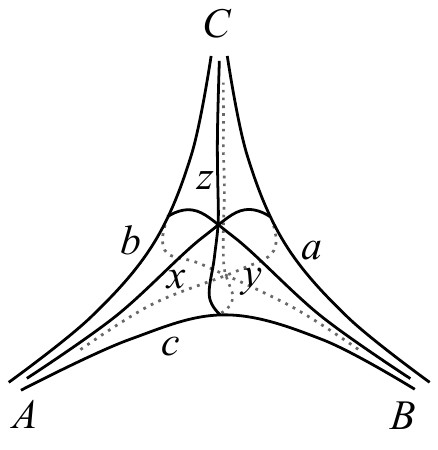}&\includegraphics[width=1.75in]{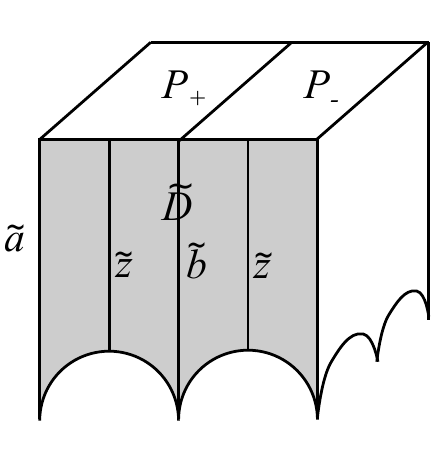}\\
(a)\textrm{ Geodesics on }D & (b)\textrm{ Lift }\widetilde{D}\textrm{ of }D
\end{array}
\]
\[
\begin{array}{c}
\includegraphics[width=2.1in]{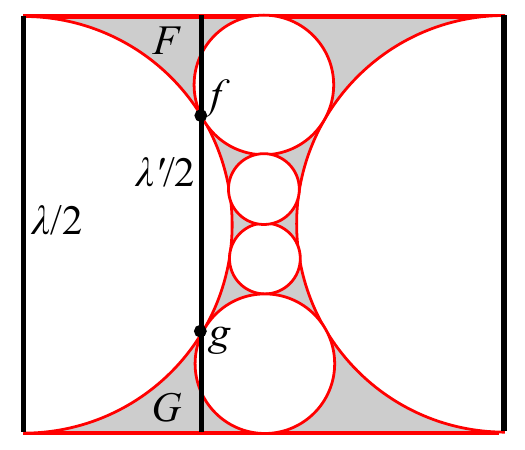}\\
(c)\textrm{ View of }\ P_+\textrm{ from }C
\end{array}
\]
\caption{Totally geodesic surfaces}
\label{fig:FixedD}
\end{figure}
\end{center}

Further, any thrice-punctured sphere can be decomposed into two ideal triangles by slicing along intercusp geodesics.  Let $T$ denote such a triangle, then the simple geodesics of the thrice-punctured sphere intersect $T$ in one of two ways.  By construction, the edges of $T$ correspond to intercusp geodesics.  The intracusp geodesics are \emph{midpoint rays}, i.e. hyperbolic rays perpendicular to one side of $T$ and pointing toward the opposite vertex.  The intracusp geodesic labeled $\tilde{z}$ in Figure \ref{fig:FixedD}$(b)$ demonstrates this phenomenon.  Since $T$ is a subset of the thrice-punctured sphere, any embedded totally geodesic surface intersects the triangle $T$ in a pairwise disjoint collection of edges and midpoint rays.  This significantly restricts how lifts of these surfaces intersect in the universal cover, a fact which we will use to our advantage.

The fundamental region $P_+ \cup P_-$ can be chosen so that the cusp corresponding to the untwisted crossing circle $C$ is at infinity.  In this case, the shaded sides are standard fundamental regions for the thrice punctured sphere $D$, and the simple geodesics on $D$ adjacent to $C$ lift to those labeled $\tilde{a}$, $\tilde{b}$ and $\tilde{z}$ in Figure \ref{fig:FixedD}$(b)$ (the intracusp geodesic is the union of the two geodesic rays labeled $\tilde{z}$).

\begin{lemma}\label{lem:FixedD}
The crossing disk $D$ is a fixed set of any symmetry of $M'_n$.
\end{lemma}

\begin{proof}
Let $\rho$ be a symmetry of $M'_n$.
As above, place the cusp corresponding to $C$ at infinity in the universal cover $\widetilde{M}'_n$ of $M'_n$.  A horosphere centered at infinity intersects the fundamental region $P_+ \cup P_-$ in a rectangle $R$ comprised of two rectangular tiles (the tile in $P_+$ for $T_6$ is depicted in Figure \ref{fig:FixedD}$(c)$).

Now $\widetilde{D}\cap R$ forms a longitude $\lambda$ for the cusp $C$ (the curve $\lambda/2$ of Figure \ref{fig:FixedD}$(c)$ represents half a longitude).  By Lemmas \ref{lem:symMaptoSelf} and \ref{lem:meridianslongitudes}, we know $\rho$ fixes $C$ and must take $\lambda$ to a copy $\lambda'$ of $\pm\lambda$.  We let $\widetilde{D}'$ denote the lift of the thrice-punctured sphere $\rho(D)$ to $P_{\pm}$. Then $\widetilde{D}'$ intersects $R$ in $\lambda'$, a parallel copy of $\lambda$. Once we show $\lambda'= \pm\lambda$ we will have that $\widetilde{D} = \widetilde{D}'$, proving that $D=D'$.

Suppose $\widetilde{D} \ne \widetilde{D}'$. We begin by showing there is, up to symmetry, only one possible $\lambda'$ different from $\lambda$.  If $\widetilde{D}'$ is different from $\widetilde{D}$, then it lifts to a vertical plane through $R$, parallel to the shaded sides and perpendicular to the white sides.  The curve $\lambda'$ cannot go precisely through the middle of $R$, for then $D'$ would have too many punctures.  Thus $\widetilde{D}'$ would have to intersect the shaded triangles labeled $F, G$ in Figure \ref{fig:FixedD}$(c)$.  Now triangles $F, G$ project to (triangles in) thrice punctured spheres $S_F, S_G$ in $M'_n$.  By the remarks preceding the lemma, the curve $\widetilde{D}'\cap F$ must be an edge or a midpoint ray of $F$. This implies $\widetilde{D}' \cap F$ must go through the vertices labeled $f, g$ in Figure \ref{fig:FixedD}$(b)$. We now finish the contradiction by showing that area considerations prevent this case from happening.

We will show that for this $\lambda'$ the area $\mathcal{A}(\widetilde{D}')$ of $\widetilde{D}'$ is greater than $2\pi$ so that it can't be a thrice punctured sphere.  Let $\widetilde{D}'_+ = \widetilde{D}'\cap P_+$ is as in Figure \ref{fig:FixedD}$(c)$, and note that $\mathcal{A}(\widetilde{D}') \ge 2 \mathcal{A}(\widetilde{D}'_+)$ because there is an identical copy of $\widetilde{D}'_+$ in $P_-$, so it suffices to show $\mathcal{A}(\widetilde{D}'_+) > \pi$.  Since $\widetilde{D}'_+$ is a polygon, its area is $2\pi$ less than the sum of the external angles, where the external angle of an ideal vertex is $\pi$.  Now $\widetilde{D}'_+$ has three ideal vertices ($f$, $g$ and at the cusp $C$), and a finite vertex along each vertical edge of $\widetilde{D}'_+$.  Summing the external angles gives a value strictly greater than $3\pi$, proving that $\mathcal{A}(\widetilde{D}'_+) > \pi$.  Thus the assumption $\widetilde{D}'\ne \widetilde{D}$ is false, proving the lemma.
\end{proof} 

\begin{lemma}\label{lem:FixedW}
The reflection surface $W$ is a fixed set of any symmetry of $M'_n$.
\end{lemma}

\begin{proof}
Again let $\rho \in Sym(M'_n)$, and we wish to show $\rho(W) = W$.  As in Figure \ref{fig:FixedD}$(a)$, label the cusps of $D$ by $A, B,$ and $C$ where $C$ corresponds to the crossing circle cusp, and label the intercusp geodesics by $a, b,$ and $c$.  The reflection surface $W$ has two components, both of which intersect $D$. One of the components, say $U$, of $W$ has just the knot circle as boundary, and can be thought of as obtained by attaching the inner and outer disks of the projection plane by one untwisted band through $C$ and $n-1$ half-twisted bands through the other crossing circles.  Thus $U$ intersects $D$ in the geodesic $c$ opposite $C$.  The other component of $W$, denoted $V$, is punctured by $C$ along two meridians, and intersects $D$ in geodesics $a$ and $b$.

Let $U'$ and $V'$ denote the images of $U$ and $V$ under $\rho$, respectively.
The isometry $\rho$ preserves the cusp $C$ by Lemma \ref{lem:symMaptoSelf}, and the crossing disk $D$ by Lemma \ref{lem:FixedD}.  Thus it either fixes or swaps cusps $A, B$, and acts analogously on intercusp geodesics opposite the cusps.  More precisely, $\rho$ preserves the geodesic $c$, and the set $\{a,b\}$.  This implies that $U'\cap D = c = U\cap D$ and $V' \cap D = \{a,b\} = V\cap D$.  Moreover, $\rho$ preserves angles so all the surfaces $U, U', V, V'$ are orthogonal to $D$. Thus both $U$ and $U'$ are connected, embedded, totally geodesic surfaces that intersect $D$ orthogonally in the geodesic $c$ (and similarly for $V, V'$ using geodesics $\{a,b\}$).  Therefore both lift to the plane in $P_+$ containing $\tilde{c}$ and orthogonal to $\widetilde{D}$.  Projecting back to $M'_n$ shows that in fact $U=U'$ (similarly $V=V'$).  Since $W = U \cup V$, the proof is complete. \end{proof}

Lemmas \ref{lem:FixedD} and \ref{lem:FixedW} allow us to determine the images of a longitude-meridian pair $(\lambda,\mu)$ under any symmetry of $M'_n$.  The rectangle $R$ in $P_+\cup P_-$ projects to a torus boundary $\partial C$ of the cusp $C$.  The crossing disk $D$ intersects $\partial C$ in a longitude $\lambda$, and the reflection surface $W$ intersects $\partial C$ in two meridians we denote $\mu_1, \mu_2$. Choosing $\mu=\mu_1$ and $\lambda$ as our generators for $\pi_1 \partial(C)$ we can now explicitly determine their possible images under $Sym(M'_n)$.

\begin{cor} \label{cor:restrictedsyms}
If $\rho\in Sym(M'_n)$, then $\rho(\lambda) = \pm\lambda$ and $\rho(\mu) = \pm\mu_i$, for some $i\in\{1,2\}$.
\end{cor}

\begin{proof}
Since $D$ is fixed by $\rho$ (Lemma \ref{lem:FixedD}), we have $D\cap N(C) = \lambda$ is fixed as well, or that $\lambda = \pm\lambda$.  Lemma \ref{lem:FixedW} shows that $W$ is fixed by $\rho$, so $\rho(\mu)\in \rho(W)\cap N(C)=W\cap N(C)$.  This implies $\rho(\mu)\in\{\pm\mu_1,\pm\mu_2\}$.
\end{proof}

At this point, we are finally able to prove Theorem \ref{thm:smallsymgroup}.

\begin{proof}[Proof of Theorem \ref{thm:smallsymgroup}] 
Fix a cusp expansion for $C$. Since any symmetry of $M_{n}'$ maps $C$ to $C$, we can consider the homomorphism $f: Sym(M_{n}') \rightarrow Sym(C)$ given by restriction. We claim that this homomorphism is injective. Let  $\rho \in Sym(M_{n}')$, and suppose $\rho$ restricted to $C$ is the identity. So, $\rho$ will fix any given point in the interior of $C$ along with a tangent frame at that point. Then Proposition A.2.1 in Benedetti--Petronio \cite{BePe1992} implies that $\rho$ must be the identity map. Thus, the kernel of $f$ is trivial, making this homomorphism injective. 

Now, Corollary \ref{cor:restrictedsyms} implies that $|Sym(C)|  \leq 8$ since there are only $8$ possible combinations for where $\lambda$ and $\mu$ could map to under a symmetry, and these symmetries of $C$ are completely determined by how they act on $\lambda$ and $\mu$. At the same time, we know that $|Sym(\mathbb{S}^{3}, \mathcal{P}_{n}')| \geq 8$, since we have already identified a set of symmetries of $(\mathbb{S}^{3}, \mathcal{P}_{n}')$ that generates a group of order $8$. Since $f$ is injective, we now have that $8 \leq |Sym(\mathbb{S}^{3}, \mathcal{P}_{n}')|  \leq |Sym(M_{n}')| \leq |Sym(C)| \leq 8$, giving the desired result. 
\end{proof}

Finally, we can prove the main result of this section. Recall that a hidden symmetry of a hyperbolic $3$-manifold $M = \mathbb{H}^{3} / \Gamma$ is an element of $C(\Gamma) / N(\Gamma)$. 

\begin{thm}
	\label{thm:lotsofHS}
	For $n \geq 5$, $M_{n}'$ is a non-arithmetic hyperbolic $3$-manifold with $2n$ hidden symmetries. 
\end{thm}

\begin{proof}
For $n \geq 5$, $M_{n} = \mathbb{H}^{3} / \Gamma_{n}$ and $M_{n}' = \mathbb{H}^{3} / \Gamma_{n}'$ are in the same commensurability class (Proposition \ref{prop:RefOrbComm}) and non-arithmetic (Theorem \ref{mainthm1}). Thus, they cover a common minimal orbifold $\mathcal{O}_{n} = M_{n} /C(\Gamma_{n})$. In Section \ref{sec:SymandHS}, we showed that $M_{n}$ has no hidden symmetries and its (full) symmetry group has order $16n$. This implies that $[C(\Gamma_{n}): \Gamma_{n}] = 16n$. Thus, the cover $M_{n} \rightarrow \mathcal{O}_{n}$ is degree $16n$, and since $M_{n}$ and $M_{n}'$ have the same volume, we also have that the cover $M_{n}' \rightarrow \mathcal{O}_{n}$ is degree $16n$. Now, Theorem \ref{thm:smallsymgroup} implies that $[N(\Gamma_{n}') : \Gamma_{n}'] = |Sym(M_{n}')| = 8$ and since $16n = [N(\Gamma_{n}') : \Gamma_{n}'] [C(\Gamma_{n}'): N(\Gamma_{n}')]$, we have that $[C(\Gamma_{n}'): N(\Gamma_{n}')] = 2n$. This implies that $M_{n}'$ has $2n$ hidden symmetries, as needed. 
\end{proof}

\begin{cor}
\label{cor:countingHS}
The number of hidden symmetries of $M_{n}'$ grows linearly with volume.
\end{cor}

\begin{proof}
Recall that $vol(M_{n}') = vol(M_{n}) =8n(\mathcal{L}(\frac{\pi}{4}+\frac{\pi}{2n}) + \mathcal{L}(\frac{\pi}{4} - \frac{\pi}{2n}))$, where $\mathcal{L}(\theta)$ denotes the Lobachevsky function. Since the Lobachevsky function is continuous, we can choose an arbitrarily small $\epsilon >0$, so that for all $n$ sufficiently large, we have
\begin{center}
$16n(\mathcal{L}(\frac{\pi}{4}) - \epsilon) \leq vol(M_{n}') \leq 16n(\mathcal{L}(\frac{\pi}{4}) + \epsilon)$.
\end{center}
 Let $HS_{n}$ denote the number of hidden symmetries for $M_{n}'$. Since Theorem \ref{thm:lotsofHS} tells us that $M_{n}'$ has $2n$ hidden symmetries, the above inequality implies that
 \begin{equation*}
  \frac{vol(M_{n}')}{8(\mathcal{L}(\pi/4) + \epsilon)} \leq  HS_{n}  \leq \frac{vol(M_{n}')}{8(\mathcal{L}(\pi/4) - \epsilon)}
 \end{equation*}   
 
which gives the desired bound. 
\end{proof}

\textbf{Remark 1:} Corollary \ref{cor:countingHS} actually highlights the fastest growth rate for the number of hidden symmetries that any sequence of non-arithmetic hyperbolic $3$-manifold can have relative to volume. To justify this, we give an upper bound. Given a non-arithmetic hyperbolic $3$-manifold $M = \mathbb{H}^{3} / \Gamma$, the number of hidden symmetries of $M$ is at most $[C(\Gamma): \Gamma]$, with equality exactly when $M$ admits no symmetries. At the same time, $[C(\Gamma): \Gamma]$ also gives the degree of the cover $M \rightarrow \mathcal{O}$, where $\mathcal{O}$ is the minimal orbifold in $M$'s commensurability class. Since $vol(\mathcal{O})$ is uniformly bounded below by the volume of the minimal volume hyperbolic $3$-orbifold (say this has volume $v_{0}$), we have that any non-arithmetic $M$ has at most $vol(M) / v_{0}$ hidden symmetries. Thus, given any sequence of non-arithmetic hyperbolic $3$-manifolds, the number of hidden symmetries can grow at most linearly with volume. 

\textbf{Remark 2:} In \cite[Section 7]{M2017}, Millichap shows that if you perform sufficiently long  Dehn fillings on all of the crossing circles of $\mathcal{P}_{n}'$ (with $n$ odd), then the resulting knot complement admits no hidden symmetries. It is interesting to see that even though $\mathcal{M}_{n}'$ admits many hidden symmetries, performing long Dehn fillings will always break these hidden symmetries in these cases. In light of this fact and Theorem 1.2 from \cite{DM2018}, it seems plausible that highly twisted knot complements (that come from performing sufficiently long Dehn fillings along crossing circles of a FAL with a single knot component) will admit no hidden symmetries.


\bibliographystyle{plain}

\begin{thebibliography}{99}

\bibitem{Adams1991}
{\scshape Adams, C.}
Limit volumes of hyperbolic three-orbifolds.
{\em J. Differential Geometry.} {\bf 34} (1991), no. 1, 115--141.
\mrev{1114455}  (92d:57029), 
\zbl{0697.57007},
\doi{10.4310/jdg/1214446993}.

\bibitem{BachShallit}
{\scshape Bach, E.; Shallit, J.}
Algorithmic number theory. Vol. 1. Efficient algorithms	
{ \em Foundations of Computing Series.} MIT Press, Cambridge, MA, 1996.
\mrev{1406794} (97e:11157).

\bibitem{B2001}
{\scshape Baker, M.D.}
Link complements and the Bianchi modular groups.  	
{ \em Trans. A. M. S.} {\bf 353} (2001), 3229--3246.
\mrev{1695016}  (2001j:57007), 
\zbl{0986.20049},
\doi{10.1090/S0002-9947-01-02555-7}

\bibitem{B2002}
{\scshape  Baker, M.D.}
All links are sublinks of arithmetic links. 	
{ \em Pacific J. Math.} {\bf 203} (2002), no. 2, 257--263. 
\mrev{1897900}  (2003a:57008), 
\zbl{1051.57007},
\doi{10.2140/pjm.2002.203.257}.

\bibitem{BR2014}
{\scshape Baker, M.D.; Reid, A.W.}
Principal Congruence Link Complements. 	
{\em Ann. Fac. Sci. Toulouse Math.} {\bf 23} (2014), no. 5, 1063--1092.
\mrev{3294602}, 
\zbl{1322.57015},
\doi{10.5802/afst.1436}.

\bibitem{BePe1992}
{\scshape Benedetti, R.; Petronio, C.}
Lectures on hyperbolic geometry.
{ \em Universitext, Springer-Verlag.} Berlin (1992).
\mrev{1219310} (94e:57015).

\bibitem{Cal2006}
{\scshape Calegari, D.}
Real places and torus bundles. 	
{\em Geom. Dedicata.} {\bf 118} (2006), 209--227.
\mrev{2239457} (2007d:57026), 
\zbl{1420.57047},
\arx{math/0510416},
\doi{10.1007/s10711-005-9037-9}.

\bibitem{CKP2018}
{ \scshape Champanerkar, A.;  Kofman, I.; Purcell, J.} 
Geometry of Biperiodic Alternating Links. 	
{\em J. Lond. Math. Soc. (2).} {\bf 99} (2019) no. 3, 807--830.
\mrev{3977891}, 
\zbl{07079420},
\arx{1802.05343},
\doi{10.1112/jlms.12195}.

\bibitem{CD2017}
{ \scshape Chesebro, E.; DeBlois, J.}
Hidden symmetries via hidden extensions.
{\em Proc. Amer. Math. Soc. 145.} {\bf 8} (2017), 3629--3644.
\mrev{3652814}, 
\zbl{1379.57009},
\arx{1501.00726}.

\bibitem{DM2018}
{ \scshape Chesebro, E.; DeBlois, J.; Mondal, P.}
Generic hyperbolic knot complements without hidden symmetries.
Preprint,
\arx{1910.04712}.

\bibitem{CDW2012}
{ \scshape Chesebro, E.; DeBlois, J.; Wilton, H.}
Some virtually special hyperbolic 3-manifold groups.	
{ \em Comment. Math. Helv.} {\bf 87} (2012), no. 3, 727--787.
\mrev{2980525}, 
\zbl{1283.57007},
\arx{0903.5288},
\doi{10.4171/CMH/267}.

\bibitem{F2017}
{ \scshape Flint, R.}
Intercusp Geodesics and Cusp Shapes of Fully Augmented Links. 	
{ \em ProQuest LLC.}  Ann Arbor, MI, 2017, Thesis (Ph.D.) - City University of New York.
\mrev{3664936},
\arx{1811.07397}.

\bibitem{Ford1999}
{ \scshape Ford, K.}
The number of solutions of $\phi(x)=m$.	
{ \em Ann. of Math. (2).} {\bf 150} (1999), no. 1, 283--311.
\mrev{1715326} (2001e:11099), 
\zbl{0978.11053},
\arx{math/9907204},
\doi{10.2307/121103}.

\bibitem{FP2007}
{ \scshape Futer, D.; Purcell, J.}
Links with no exceptional surgeries. 	
{ \em Comment. Math. Helv.} { \bf 82} (2007), no. 3, 629--664.
\mrev{2314056} (2008k:57008), 
\zbl{1134.57003},
\arx{math/0412307},
\doi{10.4171/CMH/105}.

\bibitem{HOT2018}
{ \scshape Harnois, J.; Olson, H.;  Trapp, R.}
Hyperbolic tangle surgeries and nested links.
{ \em Algebraic \& Geometric Topology.}  { \bf 18} (2018), no. 3, 1573--1602.
\mrev{3784013}, 
\zbl{1396.57008},
\doi{10.2140/agt.2018.18.1573}.

\bibitem{HW1992}
{ \scshape Henry, S.R.; Weeks, J.R.}
Symmetry groups of hyperbolic knots and links.
{ \em J. Knot Theory Ramifications.} { \bf 1} (1992), no. 2, 185--201. 
\mrev{1164115} (93e:57007), 
\zbl{0757.57008},
\doi{10.1142/S0218216592000100}.

\bibitem{J2013}
{ \scshape  Jeon, B.}
Hyperbolic 3-manifolds of bounded volume and trace field degree.	
{ \em ProQuest LLC, 2013.} Thesis (Ph.D.) - University of Illinois at Urbana-Champaign.
\mrev{3218104}.

\bibitem{L2004}
{ \scshape Lackenby, M.}
The volume of hyperbolic alternating link complements. 	
{ \em Proc. London Math. Soc. (3)} { \bf 88} (2004), no. 1, 204--224,
With an appendix by Ian Agol and Dylan Thurston.
\mrev{2018964} (2004i:57008), 
\zbl{1041.57002},
\arx{math/0012185},
\doi{10.1112/S0024611503014291}.

\bibitem{Leh}
{ \scshape  Lehmer, H.}
A note on trigonometric algebraic numbers. 	
{ \em American Mathematical Monthly.} { \bf 40} (1933), 165--166.
\mrev{1522747}, 
\zbl{0063.00001},
\doi{10.2307/2301023}.

\bibitem{MR}
{ \scshape Maclachlan C.; Reid, A.W.} 
The Arithmetic of Hyperbolic 3--Manifolds.
{ \em Graduate Texts in Mathematics,} vol. 219, Springer-Verlag, New York, (2003).
\mrev{1937957} (2004i:57021), 
\zbl{1025.57001}.

\bibitem{Ma1991}
{ \scshape Margulis, G.A.}
Discrete subgroups of semisimple Lie groups.
{\em Ergebnisse der Mathematik und ihrer Grenzgebiete (3).} {\bf 17} Springer, 1991. 
\mrev{1090825} (92h:22021), 
\zbl{0732.22008}.

\bibitem{MM2012} 
{ \scshape Marshall, T.H.; Martin, G.J.}
Minimal co-volume hyperbolic lattices, {II}: Simple torsion	in a Kleinian group.
{ \em Ann. of Math. (2).} {\bf 176} (2012), no. 1, 261--301. 
\mrev{2925384}, 
\zbl{1252.30030},
\doi{10.4007/annals.2012.176.1.4}.

\bibitem{M2017}
{ \scshape Millichap, C.}
Mutations and short geodesics in hyperbolic 3-manifolds.
{ \em Comm. Anal. Geom.} {\bf 25} (2017), no. 3, 625--683. 
\mrev{3702548}, 
\zbl{1385.57018},
\arx{1406.6033},
\doi{10.4310/CAG.2017.v25.n3.a5}.

\bibitem{MW2016}
{ \scshape Millichap, C.; Worden, W.}
Hidden symmetries and commensurability of 2-bridge link complements.
{ \em Pacific J. Math.} { \bf 285} (2016), no. 2, 453--484.
\mrev{3575575}, 
\zbl{1361.57013},
\arx{1601.01015},
\doi{10.2140/pjm.2016.285.453}.

\bibitem{N2011}
{ \scshape  Neumann, W.}
Realizing arithmetic invariants of hyperbolic 3-manifolds.
{ \em Interactions between hyperbolic geometry, quantum topology and number theory.} 233--246,
{ \em Contemp. Math., Amer. Math. Soc., Providence, RI} {\bf 541}, (2011).  
\mrev{2796636} (2012e:57030), 
\zbl{1237.57013},
\arx{1108.0062},
\doi{10.1090/conm/541}.

\bibitem{NR1990}
{ \scshape Neumann, W.; Reid, A.W.}
Arithmetic of hyperbolic manifolds. 	
{ \em Topology '90} (Columbus, OH, 1990), { \em Ohio State Univ. Math. Res. Inst. Publ., vol. 1, de Gruyter, Berlin,} 1992,  273--310.
\mrev{1184416} (94c:57024), 
\zbl{0777.57007}.

\bibitem{NT2016}
{ \scshape Neumann, W.; Tsvietkova, A.}
Intercusp geodesics and the invariant trace field of hyperbolic 3-manifolds. 	
\em{ Proc. Amer. Math. Soc.} {\bf 144 }(2016), no. 2, 887--896.
\mrev{3430862}, 
\zbl{1360.57020},
\arx{1402.5582},
\doi{10.1090/proc/12704}.

\bibitem{P2011}
{ \scshape Purcell, J.}
An introduction to fully augmented links. 	
{\em Interactions between hyperbolic geometry, quantum topology and number theory.} 205--220, 
{ \em Contemp. Math., Amer. Math. Soc., Providence, RI} {\bf 541}, (2011).  
\mrev{2796634} (2012c:57019), 
\zbl{1236.57006},
\doi{10.1090/conm/541}.

\bibitem{P2010}
{ \scshape Purcell, J.}
Hyperbolic geometry of multiply twisted knots.
{ \em Communications in Analysis and Geometry.} { \bf 18} (2010), no. 1, 101--120. 
\mrev{2660459} (2012b:57040), 
\zbl{1213.57016},
\arx{0709.2919},
\doi{10.4310/CAG.2010.v18.n1.a4}.

\bibitem{P2008}
{ \scshape Purcell, J.} 
Slope lengths and generalized augmented links.
{ \em Communications in Analysis and Geometry.} { \bf 16} (2008), no. 4, 883--905. 
\mrev{2471374} (2009k:57014), 
\zbl{1171.57009},
\arx{math/0703638}
\doi{10.4310/CAG.2008.v16.n4.a7}.

\bibitem{P2007}
{ \scshape Purcell, J.}
Volumes of highly twisted knots and links.
{ \em Algebraic \& Geometric Topology.} { \bf 7} (2007), 93--108.
\mrev{2289805} (2007m:57011), 
\zbl{1135.57005},
\arx{math/0604476},
\doi{10.2140/agt.2007.7.93}.

\bibitem{RW2008}
{ \scshape Reid, A.W.; Walsh, G.}
Commensurability classes of 2-bridge knot complements.
{ \em Algebraic \& Geometric Topology.} { \bf 8} (2008), no. 2, 1031--1057.
\mrev{2443107} (2009i:57020), 
\zbl{1154.57001},
\arx{math/0612473},
\doi{10.2140/agt.2008.8.1031}.

\bibitem{Thurston}
{ \scshape Thurston, W.}
The Geometry and Topology of 3-Manifolds.
{ \em  Lecture Notes, Princeton University Math. Dept.} (1978).

\bibitem{Thurston2}
{ \scshape Thurston, W.}
Three-dimensional manifolds, Kleinian groups and hyperbolic geometry.
{ \em Bull. Amer. Math. Soc. (N.S.)} {\bf 6} (1982), no. 3, 357--381.
\mrev{0648524} (83h:57019), 
\zbl{0496.57005},
\doi{10.1090/S0273-0979-1982-15003-0}.

\bibitem{Wa2011} 
{ \scshape Walsh, G.S.}
Orbifolds and commensurability.
{ \em Interactions between hyperbolic geometry, quantum topology and number theory.} 221--231,
{ \em Contemp. Math., Amer. Math. Soc., Providence, RI} {\bf 541}, (2011).  \mrev{2796635} (2012e:57048), 
\zbl{1231.57017},
\arx{1003.1335},
\doi{10.1090/conm/541}.

\bibitem{WatZeit1993} 
{ \scshape Watkins, W.; Zeitlin, J.}
The minimal polynomial of $\cos(2\pi/n)$.
{ \em Amer. Math. Monthly} { \bf 100} (1993), no. 5, 471--474. 
\mrev{1215534} (94b:12001), 
\zbl{0796.11008},
\doi{10.2307/2324301}.

\end{thebibliography}

\end{document}